\documentclass[a4paper,final]{amsart}

\usepackage[foot]{amsaddr}
\usepackage{amsmath}
\usepackage{amssymb}
\usepackage{pstricks}
\usepackage{graphicx,subfigure,epsfig}
\usepackage{nicefrac,bm}
\usepackage{esint} 
\usepackage{cite}
\usepackage{a4wide}

\newtheorem{theorem}{Theorem}[section]
\newtheorem{lemma}[theorem]{Lemma}

\newtheorem{proposition}[theorem]{Proposition}

\newtheorem{expl}[theorem]{Example}
\newenvironment{example}[1]{\begin{expl}\rm #1}{\end{expl}}
\newtheorem{remrk}[theorem]{Remark}
\newenvironment{remark}[1]{\begin{remrk}\rm #1}{\end{remrk}}

\def\IR{\mathbb R}

\def\IN{\mathbb N}

\newcommand{\C}{\mathcal{C}}
\newcommand{\E}{\mathcal{E}}
\newcommand{\FI}{\F_{\mathcal{I}}}
\newcommand{\FB}{\F_{\mathcal{B}}}
\newcommand{\M}{\mathcal M}

\renewcommand{\S}{\mathcal S}
\renewcommand{\c}{{\bm c}}
\newcommand{\e}{{\bm e}}

\newcommand{\wbeta}{\bm{\beta}}

\newcommand{\dist}{\mathrm{dist}}
\newcommand{\ce}{{ce}}

\renewcommand{\O}{\Omega}
\newcommand{\OO}{\O_0}

\newcommand{\oc}{\omega_{\c}}
\renewcommand{\oe}{\omega_{\e}}
\newcommand{\oce}{\omega_{\c\e}}

\newcommand{\N}[1]{\|#1\|}
\newcommand{\NN}[1]{\left\|#1\right\|}
\newcommand{\tnorm}[1]{|\!|\!|#1|\!|\!|}
\newcommand{\abs}[1]{|#1|}

\newcommand{\DG}{\mathsf{DG}}
\renewcommand{\choose}[2]{\genfrac{}{}{0pt}{}{#1}{#2}}

\newcommand{\uu}[1]{\bm#1}
\newcommand{\uuu}[1]{\bm{\mathsf{#1}}}              
\newcommand{\jmp}[1]{[\![#1]\!]}                  
\newcommand{\uuujmp}[1]{[\![#1]\!]}         
\newcommand{\mvl}[1]{\{\!\!\{#1\}\!\!\}}          

\newcommand{\wt}[1]{\widetilde{#1}}
\newcommand{\wh}[1]{\widehat{#1}}
\newcommand{\D}{\mathsf{D}}

\newcommand{\dd}{{\mathsf d}}
\newcommand{\dx}{\dd x}
\newcommand{\ds}{\dd s}

\def\hh{{\tt h}}
\def\cc{{\tt c}}

\newcommand{\T}{\mathcal{T}^{\sigma,l}}

\newcommand{\F}{\mathcal{F}}

\DeclareMathOperator{\rank}{rank}

\DeclareFontShape{OT1}{cmtt}{bx}{n}{
<5><6><7><8><9><10><10.95><12><14.4><17.28><20.74><24.88>cmttb10}{}

\title[Spectral Mixed DGFEM for Linear Elasticity]{Stability and Convergence of Spectral Mixed Discontinuous Galerkin Methods for 3D Linear Elasticity on Anisotropic Geometric Meshes}
\author[T.~P.~Wihler and M.~Wirz]{Thomas P.~Wihler \and Marcel Wirz}
\address{Mathematics Institute, University of Bern, CH-3012 Bern, Switzerland}
\email{wihler@math.unibe.ch \and marcelwirz82@gmail.com}



\begin{document}

\begin{abstract}
We consider spectral mixed discontinuous Galerkin finite element discretizations of the Lam\'e system of linear elasticity in polyhedral domains in~$\IR^3$. In order to resolve possible corner, edge, and corner-edge singularities, anisotropic geometric edge meshes consisting of hexahedral elements are applied. We perform a computational study on the discrete inf-sup stability of these methods, and especially focus on the robustness with respect to the Poisson ratio close to the incompressible limit (i.e. the Stokes system). Furthermore, under certain realistic assumptions (for analytic data) on the regularity of the exact solution, we illustrate numerically that the proposed mixed DG schemes converge exponentially in a natural DG norm.
\end{abstract}

\maketitle


\section{Introduction}
Consider an axi-parallel, open and bounded polyhedron $\Omega\subset\IR^3$, with Lipschitz boundary $\partial\Omega$, in the three-dimensional Cartesian system. Using a spectral discontinuous Galerkin finite element method (DGFEM), we shall study the numerical approximation of the following linear elasticity problem in mixed form: Find a displacement field $\bm{u}=(u_1,u_2,u_3)\in
  H^1_0(\Omega)^3$, and a pressure function~$p\in L^2_0(\Omega)$ such that
\begin{alignat}{2}
-\Delta\bm u+\nabla p&=\bm f&\quad&\text{in }\Omega,\label{eq:problem1}\\
\nabla\cdot\bm u+(1-2\nu)p&=0&&\text{in }\Omega,\label{eq:problem2}\\
\bm u &=\bm 0&&\text{on }\partial\Omega\label{eq:problem3}.
\end{alignat}
Here, $\nabla\cdot$ is the divergence operator,
$\nu\in(0,\nicefrac{1}{2}]$ is the Poisson ratio, and $\bm f\in L^2(\Omega)^3$
is an external force (scaled by $\nicefrac{2(1+\nu)}{E}$, where $E>0$ is Young's
modulus). We shall include the limit case~$\nu=\nicefrac12$ which
corresponds to the Stokes equations of incompressible fluid flow.

Elliptic boundary value problems in three-dimensional polyhedral domains are well-known to exhibit isotropic corner and anisotropic edge singularities, as well as a combination of the both; see, e.g., \cite{DaugeCostabelNicaise,MaRo10,MazzucatoNistor}. In a recent series of papers~\cite{SchoetzauSchwabWihler2009:1,SchoetzauSchwabWihler2009:2,SchoetzauSchwabWihlerNeumann} on the numerical approximation of the Poisson equation in 3d polyhedra the use of $hp$-version DG methods has been proposed (see also~\cite{MadayMarcati:2019} for eigenvalue problems with singular potentials). These schemes provide a convenient framework to resolve anisotropic edge singularities on (irregular) geometrically and anisotropically refined meshes, whilst using high-order spectral elements in the interior. Furthermore, supposing that the data is sufficiently smooth, it has been proved in~\cite{SchoetzauSchwabWihler2009:2,SchoetzauSchwabWihlerNeumann} that exponential convergence rates for $hp$-DG methods can be
achieved. 

In our previous work~\cite{wihlerWirz}, we have employed the approach~\cite{SchoetzauSchwabWihler2009:1,SchoetzauSchwabWihler2009:2,SchoetzauSchwabWihlerNeumann} in order to apply the high-order mixed DG methods introduced in~\cite{HoustonSchoetzauWihler2006} to the three-dimensional framework. More precisely, we have analyzed high-order interior penalty (IP) DG methods (of uniform but arbitrarily high polynomial degree) for the numerical approximation of~\eqref{eq:problem1}--\eqref{eq:problem3} on geometrically refined edge meshes. They can be seen as $hp$-methods with fixed and uniform polynomial degrees, or as spectral methods on locally refined meshes; in this paper they shall simply be termed spectral DGFEM. Incidentally, in contrast to classical IPDG methods, these DG schemes feature anisotropically scaled penalty terms which account for possible element anisotropies; see also~\cite{GeHaHo07}. A focal point of the article~\cite{wihlerWirz} has been to provide an inf-sup stability analysis for mixed IPDG schemes on anisotropic meshes. Our results, which are based in parts on~\cite{Schoetzau:2004:MHP}, are explicit with respect to both the (uniform) polynomial degree and the Poisson ratio~$\nu$; cf. also~\cite{MadayBernardi} in the context of spectral methods for the Stokes equations. In particular, for fixed (but arbitrarily high) polynomial degrees, our stability analysis proves that the behaviour of the mixed DG scheme remains robust as~$\nu$ tends to the critical limit of~$\nicefrac12$ of incompressible materials. Furthermore, following the techniques presented in~\cite{SchoetzauSchwabWihler2009:2,SchoetzauSchwabWihlerNeumann,SW03} we showed that the proposed DG schemes are able to achieve exponential rates of convergence for the class of piecewise analytic functions in weighted Sobolev spaces studied in~\cite{DaugeCostabelNicaise}.

The goal of the present paper is to provide a computational investigation of the theoretical inf-sup stability results on anisotropic geometric edge meshes presented in~\cite{Schoetzau:2004:MHP,wihlerWirz}. Furthermore, we will confirm the asserted exponential convergence of the spectral mixed DG method for a number of examples with typical edge and corner singularities in polyhedral domains. We will also look into the robustness of the DG approach with respect to the Poisson ratio as~$\nu\to\nicefrac12$. The precise outline of the paper is as follows:  In Section~\ref{sc:2}, we present the mixed formulation of~\eqref{eq:problem1}--\eqref{eq:problem3}, and recall its regularity in terms of  anisotropically weighted Sobolev spaces. Furthermore, in Section~\ref{sec:dgdisc} the geometric edge meshes and the spectral mixed DG discretizations will be introduced; in addition, in Section~\ref{sec:stability}, we revisit a discrete inf-sup stability framework together with the well-posedness of the DG scheme on anisotropic geometric edge meshes. Then, in Section~\ref{sec:AspectsNumComp}, we discuss the practical computation of the DG solution as well as of the discrete inf-sup constants, and present some numerical results in a few typical reference situations. In Section~\ref{sec:CompResults} we perform a number of experiments which confirm the exponential convergence as well as the robustness of the DG method with respect to the Poisson ratio. Finally, we add a few concluding remarks in Section~\ref{sec:conclusions}.

\subsubsection*{Notation}
Throughout this article, we use the following notation: For a domain
$D\subset{\IR^{d}}$, $d\ge1$, let $L^2(D)$ signify the
Lebesgue space of square-integrable functions equipped with the usual
norm~${\N{\cdot}_{L^2(D)}}$. Furthermore, we write~$L^2_0(D)$ to
denote the subspace of~$L^2(D)$ of all functions with vanishing mean
value on~$D$. The standard Sobolev space of functions with integer
regularity exponent $s \ge 0$ is signified by $H^s(D)$; we write
$\N{\cdot}_{H^s(D)}$ and $\abs{\cdot}_{H^s(D)}$ for the corresponding
norms and semi-norms, respectively. As usual, we define $H^1_0(D)$ as
the subspace of functions in $H^1(D)$ with zero trace on $\partial D$. For
vector- and tensor-valued functions we use the standard
notation $(\nabla \uu{v})_{ij}:=\partial_j v_i$, and
$(\nabla\cdot\underline{\sigma})_i:=\sum_{j=1}^3\,\partial_j \sigma_{ij}$
and $\underline{\sigma}:\underline{\tau}:=\sum_{i,j=1}^3\, \sigma_{ij}\tau_{ij}$, respectively.
Moreover, for vectors $\uu{v}, \uu{w}\in\mathbb{R}^3$, let $\uu{v}\otimes\uu{w}\in\mathbb{R}^{3\times3}$ be the matrix whose $ij$th component is $v_i\,w_j$.


\section{Linear elasticity in polyhedra}\label{sc:2}
\label{sc:problem}
We discuss the weak formulation of the mixed system of linear elasticity~\eqref{eq:problem1}--\eqref{eq:problem3}. Furthermore, we review its regularity in polyhedral domains.

\subsection{Mixed formulation and well-posedness}
A standard mixed formulation
of~\eqref{eq:problem1}--\eqref{eq:problem3} is to find $(\uu u, p)\in
H^1_0(\Omega)^3\times L^2_0(\Omega)$ such that
\begin{equation}\label{eq:weakmixed}
\begin{split}
A(\uu u,\uu v)+B(\uu v,p)&=\int_\Omega\uu f\cdot\uu v\,\dx,
\\ -B(\uu u,q) +C(p,q) &=0,
\end{split}
\end{equation}
for all $(\uu v,q)\in H^1_0(\Omega)^3\times L^2_0(\Omega)$, where
\begin{align*}
A(\uu u, \uu v)&:=\int_\Omega\nabla\uu u:\nabla\uu v\,\dx,
& B(\uu v,q)&:=-\int_\Omega q\,\nabla\cdot\uu v\,\dx,
\\ C(p,q)&:=(1-2\nu)\int_\Omega pq\,\dx.
\end{align*}
More compactly, we can write~\eqref{eq:weakmixed} equivalently in the form 
\begin{equation}\label{eq:compact}
a(\uu u,p;\uu v,q)=\int_\Omega\uu f\cdot\uu v\,\dx\qquad\forall\,(\uu v,q)\in
H^1_0(\Omega)^3\times L^2_0(\Omega),
\end{equation}
with
\[
a(\uu u,p;\uu v,q):=A(\uu u,\uu v)+B(\uu v,p)-B(\uu u,q)+C(p,q).
\]
It is straightforward to verify that~$a$ is a bounded bilinear form on~$H^1_0(\Omega)^3\times L^2_0(\Omega)$, and that, for $\nu\in(0,\nicefrac{1}{2})$, it is also coercive. In particular, by application of the Lax-Milgram theorem, we conclude that the solution~$(\uu u,p)\in H^1_0(\Omega)^3\times L^2_0(\Omega)$
of~\eqref{eq:compact}, and, thus, of~\eqref{eq:weakmixed}, exists and is
unique. In addition, for~$\nu=\nicefrac12$, which corresponds to the
Stokes equations, problem~\eqref{eq:compact} is still well-posed. Indeed, this is an immediate consequence of the inf-sup condition
\[
\inf_{0\not\equiv q\in L^2_0(\Omega)}\sup_{\uu 0\not\equiv \uu{v}\in
H^1_0(\Omega)^3}
\frac{-\int_\Omega\, q\nabla\cdot\uu{v}
\,\dx}{\|\nabla\uu{v}\|_{L^2(\Omega)} \|q\|_{L^2(\Omega)}}\ge \kappa
>0,
\]
where $\kappa$ is the inf-sup constant, depending only on $\Omega$;
see~\cite{BrezziFortin91,GiraultRaviart86} for
details.

\subsection{Regularity}
Following~\cite{DaugeCostabelNicaise} we recall the regularity of the solution of~\eqref{eq:problem1}--\eqref{eq:problem3} in weighted Sobolev spaces; cf.~also~\cite{Guo95,BabGuo_edge1,BabGuo_edge2}. To this end,
we denote by~$\C$ the set of corners, and by~$\E$ the set of
edges of~$\Omega$. Potential singularities of the solution are located on the skeleton of~$\Omega$ given by
\[
\S=\left(\bigcup_{\c\in\C}\c\right)\cup\left(\bigcup_{\e\in\E}\e\right)\subset \partial\Omega.
\]
Given a corner~$\c\in\C$, an edge~$\e\in\E$, and
$x\in\Omega$, we define the distance functions
\[
r_{\c}(x)=\dist(x, \c), \qquad r_\e(x)=\dist(x,
\e),\qquad \rho_{\ce}(x)=r_\e(x)/r_{\c}(x).
\]
Furthermore, for each corner~$\c\in\C$, we signify by
$\E_{\c}=\left\{\,\e\in\E\,:\, \c\cap\overline\e\neq\emptyset\,\right\}$
the set of all edges of~$\Omega$ which meet at~$\c$.
For any $\e\in\E$, the set of corners of~$\e$ is given by
$\C_\e=\left\{\,\c\in\C\,:\,\,\c\cap\overline\e\neq\emptyset\,\right\}$.
Then, for $\c\in\C$, $\e\in\E$, respectively $\e \in\E_{\c}$, and a sufficiently small parameter~$\varepsilon>0$,
we define the neighbourhoods
\begin{equation*}
\begin{split}
\oc&=\{\,x\in\Omega\,:\, r_{\c}(x)<\varepsilon \;\wedge\;
\rho_{\ce}(x)>\varepsilon \quad\forall\,\e\in\E_{\c}\,\}, \\
\oe&=\{\,x\in\Omega\,:\,\, r_\e(x)<\varepsilon \;\wedge\;
r_{\c}(x)>\varepsilon \quad\forall\,\c\in\C_\e\,\}, \\
\oce&=\{\,x\in\Omega\,:\, r_{\c}( x)<\varepsilon \;\wedge\;
\rho_{\c\e}(x)<\varepsilon\,\}.
\end{split}
\end{equation*}
Moreover, we define the interior part of~$\Omega$ by $\Omega_0=\{x\in\Omega:\,\dist(x,\partial\Omega)>\varepsilon\}$.

Near corners $\c\in\C$ and edges $\e\in\E$, we shall use local coordinate systems in~$\oe$ and $\oce$, which are chosen such that $\e$ corresponds to the direction~$(0,0,1)$.  Then, we denote quantities that are transversal to~$\e$ by~$(\cdot)^\perp$, and quantities parallel to~$\e$ by~$(\cdot)^\|$. For instance, if~$\bm\alpha\in\mathbb{N}^3_0$ is a multi-index associated with the three local coordinate directions in~$\oe$ or~$\oce$, then we write~$\bm\alpha=(\bm \alpha^\perp, \alpha^\|)$, where
$\bm \alpha^\perp=(\alpha_1, \alpha_2)$
and~$\alpha^\|=\alpha_3$. In addition, we use the notation~$|\bm\alpha^\perp|=\alpha_1+\alpha_2$, and~$|\bm\alpha|=|\bm\alpha^\perp|+\alpha^\|$.

Following~\cite[Def.~6.2 and Eq.~(6.9)]{DaugeCostabelNicaise}, we introduce anisotropically weighted Sobolev spaces. To this end, to each $\c\in\C$ and $\e\in\E$, we associate a corner and an edge exponent $\beta_\c, \beta_\e\in\mathbb{R}$, respectively.
We collect these quantities in the weight vector $\wbeta=\{\beta_\c:\,\c\in\C\}\cup\{\beta_\e:\,\e\in\E\} \in \mathbb{R}^{|\C|+|\E|}$. Then, for $m\in\mathbb{N}_0$, we define the weighted semi-norm
\begin{equation*}
\begin{split}
|v|^2_{M^m_{\wbeta}(\Omega)} = |v|^2_{H^m(\OO)}&+
\sum_{\e\in\E}\sum_{\genfrac{}{}{0pt}{}{\bm\alpha\in\mathbb{N}^3_0}{|\bm\alpha|=m}}
\big\|{r_\e^{\beta_\e+|\bm\alpha^\perp|}\D^{\bm\alpha}v}\big\|^2_{L^2(\oe)}
\\
&+\sum_{\c\in\C}\sum_{\genfrac{}{}{0pt}{}{\bm\alpha\in\mathbb{N}^3_0}{|\bm\alpha|=m}}
\big\|{r_\c^{\beta_\c+|\bm\alpha|}\D^{\bm\alpha}v}\big\|^2_{L^2(\oc)}\\
&+\sum_{\c\in\C}\sum_{\e\in\E_\c}
   \big\|{r_\c^{\beta_\c+|\bm\alpha|}\rho_{\c\e}^{\beta_\e+|\bm\alpha^\perp|}
       \D^{\bm\alpha}v}\big\|^2_{L^2(\oce)},
\end{split}
\end{equation*}
as well as the full norm  $\NN{v}^2_{M^m_{\wbeta}(\Omega)}=\sum_{k=0}^m\N{v}^2_{M^k_{\wbeta}(\Omega)}$; here, the operator $\D^{\bm\alpha}$ denotes
the partial derivative in the local coordinate directions corresponding to
the multi-index~$\bm\alpha$. The space $M^m_{\wbeta}(\Omega)$ is the weighted Sobolev space obtained as the closure of~$C^\infty_0(\Omega)$ with respect to the norm~$\NN{\cdot}_{M^m_{\wbeta}(\Omega)}$. 

We notice the following regularity property of the solution
of~\eqref{eq:problem1}--\eqref{eq:problem3} in terms of the weighted
Sobolev spaces defined above; see~\cite{DaugeCostabelNicaise} (in
addition, cf.~\cite{MaRo10,MazzucatoNistor}):

\begin{proposition}\label{pr:LowRegularity}
There exist upper bounds $\beta_{\mathcal E}, \beta_{\mathcal C}>0$ such that, if the weight vector~$\bm\beta$ satisfies
\[
0 < \beta_{\bm e} < \beta_{\mathcal E},\quad 
0 < \beta_{\bm c} < \beta_{\mathcal C},
\qquad \bm e\in{\mathcal E},\ \bm c\in{\mathcal C},
\]
then, for every~$m\in\mathbb{N}$, the solution~$(\bm u,p)\in
H_0^1(\Omega)^3\times L^2_0(\Omega)$
of~\eqref{eq:problem1}--\eqref{eq:problem3} with~$\bm f\in
M^m_{1-\bm\beta}(\Omega)^3$ fulfills~$(\bm u,p)\in
M^m_{-1-\bm\beta}(\Omega)^3\times M^{m-1}_{-\bm\beta}(\Omega)$.
\end{proposition}


\section{Mixed discontinuous Galerkin methods on geometric meshes}\label{sec:dgdisc}

In the following section we will introduce spectral mixed DG discretizations on geometric meshes for the numerical solution of~\eqref{eq:problem1}--\eqref{eq:problem3}.

\subsection{Hexahedral geometric edge meshes}\label{sc:meshes}
In order to numerically resolve possible corner and edge singularities in the solution~$(\bm u,p)$ of~\eqref{eq:problem1}--\eqref{eq:problem3}, we employ
anisotropic geometric edge meshes. To this end, we follow the construction in~\cite{Schoetzau:2004:MHP}, where such meshes have been studied in the context of DGFEM for the Stokes equations; see also the earlier paper~\cite{BG96} on conforming~$hp$-version finite element methods. Specifically, we begin from a coarse regular and shape-regular, quasi-uniform partition $\mathcal{T}^0=\{Q_j\}_{j=1}^J$ of~$\Omega$ into~$J$ convex axi-parallel hexahedra. Each of these elements~$Q_j\in\mathcal{T}^0$ is the image under an affine mapping~$G_j$ of the reference patch~$\wt{Q}=(-1,1)^3$, i.e. $Q_{j} = G_{j}(\wt{Q})$. The mappings~$G_j$ are compositions of (isotropic) dilations and translations.

Based on the coarse partition (macro mesh)~$\mathcal{T}^0$ we will use three canonical geometric refinements (patches) towards corners, edges and corner-edges of~$\wt{Q}$; see Figure~\ref{fig:Meshes}. They feature a refinement ratio~$\sigma\in(0,1)$, as well as a number of refinement levels~$\ell\in\mathbb{N}_0$; to give an example, in Figure~\ref{fig:Meshes}, we have selected~$\sigma=\nicefrac12$, and~$\ell=3$.

\begin{figure}
\begin{center}
\begin{pspicture}(0,-0.2)(8,4.5)
\psset{linewidth=.01}
\def\fine{0.01}
\rput(4,3){
\psline(0,0)(1,0) \psline(0,.25)(1,.25) \psline(0,.5)(1,.5)
\psline(0,.75)(1,.75) \psline(0,1)(1,1)
\psline(0,0)(0,1) \psline(.25,0)(.25,1) \psline(.5,0)(.5,1)
\psline(.75,0)(.75,1) \psline(1,0)(1,1)
\psline(1,0)(1.5,.5) \psline(1,.25)(1.5,.75) \psline(1,.5)(1.5,1)
\psline(1,.75)(1.5,1.25) \psline(1,1)(1.5,1.5)
\psline(1.125,.125)(1.125,1.125) \psline(1.25,.25)(1.25,1.25)
\psline(1.375,.375)(1.375,1.375) \psline(1.5,.5)(1.5,1.5)
\psline(.125,1.125)(1.125,1.125) \psline(.25,1.25)(1.25,1.25)
\psline(.375,1.375)(1.375,1.375) \psline(.5,1.5)(1.5,1.5)
\psline(0,1)(.5,1.5) \psline(.25,1)(.75,1.5) \psline(.5,1)(1,1.5)
\psline(.75,1)(1.25,1.5)
}
\rput(1.5,0){
\psline(0,0)(1,0)  \psline(0,.5)(1,.5) \psline(.5,.75)(1,.75)
\psline(.75,.875)(1,.875) \psline(0,1)(1,1)
\psline(0,0)(0,1) \psline(.5,0)(.5,1) \psline(.75,.5)(.75,1)
\psline(.875,.75)(.875,1) \psline(1,0)(1,1)
\psline(1,0)(1.5,.5) \psline(1,.5)(1.5,1) \psline(1,.75)(1.5,1.25)
\psline(1,.875)(1.5,1.375) \psline(1,1)(1.5,1.5)
\psline(1.5,.5)(1.5,1.5)
\psline(.5,1.5)(1.5,1.5)
\psline(0,1)(.5,1.5) \psline(.5,1)(1,1.5) \psline(.875,1)(1.375,1.5)
\psline(.75,1)(1.25,1.5)
}
\rput(4,0){
\psline(0,0)(1,0)  \psline(0,.5)(1,.5) \psline(.5,.75)(1,.75)
\psline(.75,.875)(1,.875) \psline(0,1)(1,1)
\psline(0,0)(0,1) \psline(.5,0)(.5,1) \psline(.75,.5)(.75,1)
\psline(.875,.75)(.875,1) \psline(1,0)(1,1)
\psline(1,0)(1.5,.5) \psline(1,.5)(1.5,1) \psline(1,.75)(1.25,1)
\psline(1,.875)(1.125,1) \psline(1,1)(1.5,1.5)
\psline(1.0625,.8125)(1.0625,1.0625) \psline(1.125,.625)(1.125,1.125)
\psline(1.25,.25)(1.25,1.25) \psline(1.5,.5)(1.5,1.5)
\psline(.8125,1.0625)(1.0625,1.0625) \psline(.625,1.125)(1.125,1.125)
\psline(.25,1.25)(1.25,1.25) \psline(.5,1.5)(1.5,1.5)
\psline(0,1)(.5,1.5) \psline(.5,1)(1,1.5) \psline(.75,1)(1,1.25)
\psline(.875,1)(1,1.125)
}
\rput(6.5,0){
\psline(0,0)(1,0)  \psline(0,.5)(1,.5) \psline(0,.75)(1,.75)
\psline(0,.875)(1,.875) \psline(0,1)(1,1)
\psline(0,0)(0,1) \psline(.5,0)(.5,1) \psline(.75,0)(.75,1)
\psline(.875,0)(.875,1) \psline(1,0)(1,1)
\psline(1,0)(1.5,.5) \psline(1,.5)(1.5,1) \psline(1,.75)(1.5,1.25)
\psline(1,.875)(1.5,1.375) \psline(1,1)(1.5,1.5)
\psline(1.0625,.0625)(1.0625,1.0625) \psline(1.125,.125)(1.125,1.125)
\psline(1.25,.25)(1.25,1.25) \psline(1.5,.5)(1.5,1.5)
\psline(.0625,1.0625)(1.0625,1.0625) \psline(.125,1.125)(1.125,1.125)
\psline(.25,1.25)(1.25,1.25) \psline(.5,1.5)(1.5,1.5)
\psline(0,1)(.5,1.5) \psline(.5,1)(1,1.5) \psline(.75,1)(1.25,1.5)
\psline(.875,1)(1.375,1.5)
}
\psset{linewidth=.01}
\psline{<-}(3.7,3.5)(2.75,2)
\psline{<-}(4.75,2.75)(4.75,2)
\psline{<-}(5.8,3.5)(6.75,2)
\uput[0](-1,4){\small{macro mesh~$\mathcal{T}^0$}}
\uput[0](-1,1){\small{patches}}
\end{pspicture}
\end{center}
\caption{Canonical geometric refinements (patches) towards edges, corners, and corner-edges (bottom) to be built into the macro mesh~$\mathcal{T}^0$ (top) by means of suitable affine transformations.}
\label{fig:Meshes}
\end{figure}
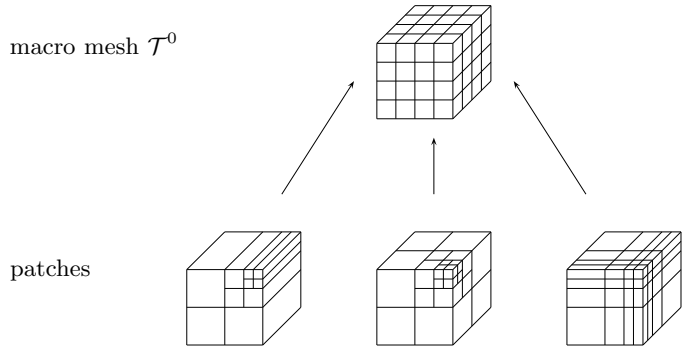

Given a (fixed) refinement ratio~$\sigma\in(0,1)$ as well as a refinement level value~$\ell\in\IN_0$, geometric meshes in~$\Omega$ are now built by applying the patch mappings~$G_j$ to transform the above canonical geometric mesh patches on the reference patch~$\wt Q$ to the macro-elements~$Q_j\in\mathcal{T}^0$, thereby yielding a local patch mesh~$\M^{\sigma,\ell}_j$ on~$Q_j$. The patches~$Q_j$ away from the singular support~$\S$ (i.e. with $\overline{Q}_j\cap \S=\emptyset$) are left unrefined, i.e. in this case we let~$\M^{\sigma,\ell}_j=\{Q_j\}$. It is important to note that the geometric refinements in the canonical patches have to be suitably selected, oriented and combined in order to achieve a proper geometric refinement towards corners and edges of~$\Omega$. Then, a $\sigma$-geometric mesh in~$\Omega$ is given by $\T=\bigcup_{j=1}^{J}\M^{\sigma,\ell}_j$. Furthermore, the sequence~$\{\T\}_{\ell\in\IN_0}$ is referred to as a $\sigma$-geometric mesh family. We note that this family of meshes is anisotropic as well as irregular. For a more general construction of geometric meshes on polyhedral domains, we refer to~\cite{SchoetzauSchwabWihler2009:1}.

\subsection{Faces and face operators}
We denote the set of all interior faces in~$\T$ by $\FI(\T)$, and the set of
all boundary faces by~$\FB(\T)$. Further, let~$\F(\T) = \FI(\T)
\cup \FB(\T)$ signify the set of all (smallest) faces of~$\T$.
In addition, for an element~$K\in\T$, we denote the set of its faces by
$\F_K=\left\{\,f\in\F\,:\,f\subset\partial K\,\right\}$.

Next, we recall the standard DG trace operators. For this purpose, consider an interior face $f=\partial K^\sharp\cap\partial K^\flat\in\FI(\T)$
shared by two neighbouring elements~$K^\sharp, K^\flat\in\T$. Furthermore, let $u$, $\bm v$ and $\underline{w}$ be scalar-, vector, and tensor-valued functions, respectively, all sufficiently smooth  inside the elements~$K^\sharp, K^\flat$. Then, we
define the following trace operators along~$f$:
\begin{align*}
\jmp{u}&=u|_{K^\sharp} \bm n_{K^\sharp}+u|_{K^\flat} \bm n_{K^\flat},&
\mvl{u}&=\nicefrac12\left(u|_{K^\sharp}+u|_{K^\flat}\right),\\
\jmp{\bm v}&=\bm v|_{K^\sharp}\cdot \bm n_{K^\sharp}+\bm v|_{K^\flat}\cdot \bm n_{K^\flat},&
\mvl{\bm v}&=\nicefrac12\left(\bm v|_{K^\sharp}+\bm v|_{K^\flat}\right),\\
\jmp{\underline{w}}&=\underline{w}|_{K^\sharp}\otimes \bm n_{K^\sharp}+\underline{w}|_{K^\flat}\otimes \bm n_{K^\flat},&
\mvl{\underline{w}}&=\nicefrac12\left(\underline{w}|_{K^\sharp}+\underline{w}|_{K^\flat}\right).
\end{align*}
Here, for an element~$K\in\T$, we denote by~$\bm n_K$ the outward
unit normal vector on~$\partial K$. Similarly, for a boundary
face~$f=\partial K\cap\partial \Omega\in\FB(\T)$, with $K\in\T$, and a sufficiently smooth scalar function~$u$, we let $\jmp{u}=u|_K \bm n_\Omega$, and $\mvl{u}=u|_K$, where $\bm n_\Omega$ is the outward unit normal vector on~$\partial\Omega$; obvious modifications are made for vector- and tensor-valued functions in accordance with the definition above.

Finally, $\nabla_h$ and $\nabla_h\cdot$ denote the element-wise
gradient and divergence operators, respectively. Here and in the sequel, we use abbreviations like
\[
  \int_{\F}(\cdot)\, \ds := \sum_{f\in{\F}}\int_f(\cdot)\, \ds,\qquad
\N{\nabla_h
    (\cdot)}^2_{L^2(\Omega)}:=\sum_{K\in\T}\N{\nabla(\cdot)}_{L^2(K)}^2.
\]

\subsection{Spectral DG discretizations}\label{sec:dgdisc1}
Given a geometric edge mesh $\T$ on $\Omega$ and a polynomial degree $k\ge1$ (which is assumed uniform and isotropic on~$\T$), we approximate~\eqref{eq:problem1}--\eqref{eq:problem3} by finite element
functions $(\uu{u}_h,p_h)\in \uu{V}_h\times Q_h$, where
\begin{equation}\label{eq:DGspaces}
\begin{split}
\uu{V}_h&:=\{\, \uu{v}\in L^2(\Omega)^3 : \uu{v}|_K \in {\mathcal Q}_{k}(K)^3, \
K\in\T \,\},\\
Q_h&:=\{\, q\in L^2_0(\Omega) : q|_K \in {\mathcal Q}_{k-1}(K),\ K\in\T \, \}.
\end{split}
\end{equation}
Here, for $k\ge 0$, $K\in\T$, ${\mathcal Q}_{k}(K)$ denotes the space
of all polynomials of degree at most $k$ in each variable on $K$. In addition, we let
\[
\bm V(h)=\bm V_h+H^1_0(\Omega)^3.
\]

On the space~$\bm V_h$ we consider the stabilization function $\cc\in L^{\infty}(\F)$ given by
\begin{equation}
\cc(x) :=\gamma\hh^{-1}(x)k^2, \label{eq:c}
\end{equation}
where $\gamma>0$ is a penalty parameter independent of the refinement ratio~$\sigma$, the number of refinement levels~$\ell$, and the polynomial degree~$k$. Furthermore, for~$x\in f$, with $f\in\F$, the mesh function~$\hh$ is defined by
\[
\hh(x) :=
\begin{cases}
 \min\{h_{K^\sharp,f}^{\perp},h_{K^\flat,f}^{\perp}\}, & x\in
f \subset {\F}_{\mathcal I}, f=\partial K^\sharp\cap\partial K^\flat, \text{with }K^\sharp, K^\flat\in\T,\\
 h_{K,f}^{\perp}, & x\in f\subset{\F}_{\mathcal B}, f=\partial K\cap\partial\Omega, \text{with }K\in\T.
\end{cases}
\]
In this definition, for~$K\in\T$ and~$f\in\F_K$, we denote by $h_{K,f}^\perp$
the diameter of the element~$K$ in the direction perpendicular to the
face $f$.  

Then, we consider the following mixed discontinuous Galerkin discretization of~\eqref{eq:weakmixed}: Find $(\uu{u}_h,p_h)\in \uu{V}_h\times Q_h$ such that
\begin{equation}\label{eq:mixedDG}
\begin{split}
A_h(\uu u_h,\uu v)+B_h(\uu v,p_h)&=\int_\Omega\uu f\cdot\uu v\,\dx , \\
-B_h(\uu u_h,q) +C_{h}(p_h,q) &=0,
\end{split}
\end{equation}
for all $(\uu{v},q)\in \uu{V}_h\times Q_h$. The forms $A_h$, $B_h$, and
$C_{h}$ are given, respectively, by
\begin{align}
A_{h}(\uu{u},\uu{v})
&:=\int_\Omega\nabla_h \uu{u}:\nabla_h\uu{v}\,\dx
-\int_{\F}\Big(\theta\mvl{\nabla_h\uu{v}}:\uuujmp{\uu{u}}+
\mvl{\nabla_h\uu{u}}:\uuujmp{\uu{v}}\Big)\, \ds \nonumber\\
&\quad + \int_{\F}\cc\,\uuujmp{\uu{u}}:\uuujmp{\uu{v}}\, \ds,
\nonumber \\
B_h(\uu{v},q)&:=
-\int_\Omega \, q\, \nabla_h\cdot\uu{v} \, \dx
+\int_{\F} \mvl{q} \jmp{\uu{v}} \, \ds,
\label{eq:DGforms} \\
C_h(p,q)&:=(1-2\nu)\int_\Omega pq\,\dx, \nonumber
\end{align}
where $\theta\in[-1,1]$ is a fixed parameter. Different choices of~$\theta$ refer to various types of interior penalty DG methods: for instance, the form $A_h$ may be chosen to correspond to the symmetric (for $\theta=1$), incomplete (for $\theta=0$), or non-symmetric (for $\theta=-1$) interior penalty DG discretization of the Laplacian; for a detailed review on a wide class of DG methods for the Poisson problem and the Stokes system, we refer to the articles~\cite{ArnoldBrezziCockburnMarini2001,Schoetzau:2001:MHP}, respectively. 

As in~\eqref{eq:compact}, the discrete DG formulation~\eqref{eq:mixedDG}
is equivalent to finding $(\uu{u}_h,p_h)\in \uu{V}_h\times Q_h$ such
that
\begin{equation}\label{eq:disccompact}
a_h(\uu u_h, p_h; \uu v,q)=\int_\Omega\uu f\cdot\uu v\,\dx
\end{equation}
for all $(\uu v,q)\in \uu V_h\times Q_h$, where
\begin{equation}
a_h(\uu u, p; \uu v,q):=A_h(\uu u,\uu v)+B_h(\uu v,p)-B_h(\uu u,q)+C_h(p,q).
\label{eq:formah}
\end{equation}

\subsection{Discrete inf-sup stability and well-posedness}\label{sec:stability}
In this section we recapitulate an inf-sup stability result from~\cite{wihlerWirz} for the form~$a_h$ given in~\eqref{eq:formah}. We first define the DG-norm
\begin{equation}\label{eq:DGNorm}
\begin{split}
\tnorm{(\uu{v},q)}^2_{\DG}:=&\N{\bm v}_h^2
+(2-2\nu)\N{q}_{L^2(\Omega)}^2,
\end{split}
\end{equation}
for any $(\uu v,q)\in\uu{V}(h)\times L^2(\Omega)$, where
\[
\N{\bm v}_h^2:=\N{\nabla_h\uu{v}}^2_{L^2(\Omega)}+
\int_{\F}\cc\,\abs{\uuujmp{\uu{v}}}^2\, \ds,\qquad \bm v\in\bm V(h).
\] 
If the Poisson ratio satisfies~$\nu\in(0,\nicefrac{1}{2})$, and provided that the penalty parameter~$\gamma$ featured in~\eqref{eq:c} is chosen sufficiently large, then the following coercivity estimate can be shown:
\[
a_h(\uu u, p;\uu u, p) \ge C\N{\uu u}_h^{2}+(1-2\nu)\N{p}_{L^2(\Omega)}^{2}\ge
C(1-2\nu)\tnorm{(\uu u,p)}^2_{\DG}\qquad\forall (\bm u,p)\in \bm V_h\times Q_h;
\]
here, $C>0$ is a constant independent of $\nu$, $k$, $l$, and the aspect
ratio of the anisotropic elements.

This result can be made stronger if a discrete inf-sup
condition on the form $B_h$ on geometric edge meshes is assumed: Let $\T$ be a geometric edge mesh on $\Omega$ as defined in Section~\ref{sc:meshes}, with refinement ratio~$\sigma \in (0,1)$, and $\ell\ge1$ layers of refinement. Suppose that there exist constants $\varkappa>0$ and~$\rho\ge 0$ that may depend on $\sigma$, $\gamma$, and on the macro-element mesh $\mathcal T^0$, but are independent of $k$, $\ell$, and the aspect ratio of the anisotropic elements in $\T$, such that there holds
\begin{equation}\label{eq:disc_infsup}
 \gamma_B:= \inf_{0 \not\equiv q \in Q_h} \sup_{\uu 0\not\equiv \uu v \in
\uu V_h}
  \frac{B_h(\uu v,q)}{\N{\uu v}_h\N{q}_{L^2(\Omega)}} \ge \varkappa k^{-\rho},
\end{equation}
as~$k\to\infty$. 

\begin{remark}\label{rm:rho}
In~\cite{Schoetzau:2004:MHP} it was proved that this assumption is fulfilled with $\rho=\nicefrac{3}{2}$ (and any~$k\ge2$). Our numerical computations in Section~\ref{sec:CompInfSupA} below indicate, however, that the dependence of the right-hand side of~\eqref{eq:disc_infsup} on $k$ is much weaker than $k^{-\nicefrac{3}{2}}$.
\end{remark}

The following result, which implies the well-posedness of~\eqref{eq:mixedDG} even in the incompressible limit~$\nu=\nicefrac12$, follows immediately from~\cite[Theorem~5.1]{wihlerWirz}.

\begin{theorem}\label{Thm:disc_infsupa}
Let~$\nu\in(0,\nicefrac12]$. If~\eqref{eq:disc_infsup} holds true, then we have the inf-sup condition
\begin{equation}\label{eq:disc_infsupa}
\gamma_a:=\inf_{\choose{(\bm u,p)\in\bm V_h\times Q_h}{(\bm u,p)\neq(\bm 0,0)}}
\sup_{\choose{(\bm v,q)\in\bm V_h\times Q_h}{(\bm v,q)\neq(\bm 0,0)}}
\frac{a_h(\uu u, p; \uu v, q)}{\tnorm{(\uu u,p)}_{\DG} 
\tnorm{(\uu v,q)}_{\DG}} \ge C\max\{k^{-2\rho},1-2\nu\},
\end{equation}
with a constant $C>0$ that depends on the penalty parameter~$\gamma$, however, is independent of $\nu$, $k$, $l$, and the aspect ratio of the anisotropic elements.
\end{theorem}

We emphasize that, for fixed~$k$, the stability bound~\eqref{eq:disc_infsupa} does not deteriorate as~$\nu\to\nicefrac12$.


\section{Computing the DG solution and the inf-sup constants}\label{sec:AspectsNumComp}

Our goal is to investigate the behavior of the inf-sup conditions from~\eqref{eq:disc_infsup} and~\eqref{eq:disc_infsupa} numerically. We note that both of them involve the discrete space~$Q_h$ from~\eqref{eq:DGspaces}. Due to the \emph{global} zero mean constraint contained in~$L^2_0(\Omega)$, the construction of~$Q_h$ in terms of \emph{standard local} basis functions, as provided by most finite element packages, causes difficulties. The classical remedy is to use the full space
\begin{equation}\label{eq:wtQ}
\wt{Q}_h:=\{\, q\in L^2(\Omega) : q|_K \in {\mathcal Q}_{k-1}(K),\ K\in\T \, \},
\end{equation}
and to impose the zero mean condition by means of a Lagrange multiplier technique. Noticing that $\dim(Q_h) = \dim(\wt{Q}_h)-1$,
we emphasize that this approach is of equivalent computational cost as the original DG system~\eqref{eq:mixedDG}, yet, it allows to employ standard discrete spaces. This, in turn, leads to a more convenient practical framework for the computation of the DG solution and the evaluation of inf-sup constants.

\subsection{Reformulation of the mixed DG discretization}\label{sec:ReformulatedSyst}

We rewrite the original system~\eqref{eq:mixedDG}, which is based on the discrete DG space~$\bm V_h\times Q_h$, on the new space~$\bm V_h\times\wt Q_h\times\mathbb{R}$, where~$\wt Q_h$ is the full space from~\eqref{eq:wtQ}. To this end, we introduce an auxiliary variable $\wt{r}\in\mathbb{R}$ which takes the role of the mean value of the pressure~$p_h$ on~$\Omega$. More precisely, let us consider the following augmented DG formulation: Find $(\wt{\uu u}_h, \wt{p}_h, \wt{r})\in
\uu V_h\times \wt{Q}_h\times \IR$ such that
\begin{equation}\label{eq:reform_mixedDG}
\begin{split}
A_h(\wt{\uu u}_h,\wt{\uu v})+B_h(\wt{\uu v},\wt{p}_h)&=\int_\Omega\uu f\cdot\wt{\uu
v}\,\dx, \\
-B_h(\wt{\uu u}_h,\wt{q}) + C_{h}(\wt{p}_h,\wt{q}) -
\wt{r}\int_{\Omega}\wt{q}\,\dx&=0,\\
\wt{s}\fint_{\Omega}\wt{p}_h\,\dx - \wt{r}\wt{s} &= 0,
\end{split}
\end{equation}
for all $(\wt{\uu v}, \wt{q}, \wt{s})\in \uu V_h\times \wt{Q}_h\times \IR$. Here
we use the notation
\begin{equation*}
\fint_{\Omega}(\cdot)\,\dx:=\frac{1}{\abs{\Omega}}\int_{\Omega}(\cdot)\,\dx
\end{equation*}
to denote the mean value integral on~$\Omega$. Note also that this new system may be written in a more compact way: Find
$(\wt{\uu u}_h,\wt{p}_h, \wt{r})\in \uu V_h\times \wt{Q}_h\times \IR$ such that
\[
\wt a_h(\wt{\uu u}_h, \wt{p}_h, \wt{r}; \wt{\uu v},\wt{q},\wt{s})=\int_\Omega\uu
f\cdot\wt{\uu v}\,\dx\qquad
\forall (\wt{\uu v}, \wt{q}, \wt{s})\in \uu V_h\times \wt{Q}_h\times \IR,
\]
where
\begin{align*}
\wt a_h(\wt{\uu u}_h, \wt{p}_h, \wt{r}; \wt{\uu v},\wt{q},\wt{s}) &:= A_h(\wt{\uu
u}_h,\wt{\uu v})+B_h(\wt{\uu v},\wt{p}_h)-B_h(\wt{\uu u}_h,\wt{q}) +
C_{h}(\wt{p}_h,\wt{q})\nonumber\\
&\quad - \wt{r}\int_{\Omega}\wt{q}\,\dx + \wt{s}\fint_{\Omega}\wt{p}_h\,\dx -
\wt{r}\wt{s}.
\end{align*}
We again stress the fact that this system can be expressed in terms of standard local basis functions, and, thereby, permits to apply a straightforward implementational setting.                                                                                                                                                                                                                                                      

To show the equivalence of the two formulations~\eqref{eq:mixedDG}
and~\eqref{eq:reform_mixedDG}, we require the following lemma. Here, we shall denote by $\mathcal Q_0(\Omega)\simeq\mathbb{R}$ the space of all (globally) constant functions on $\Omega$, and we note that
\begin{equation}\label{eq:quot}
\wt Q_h=Q_h\oplus\mathcal{Q}_0.
\end{equation}

\begin{lemma}\label{lem:kernelB}
The form $B_h$ from~\eqref{eq:DGforms} satisfies 
\begin{equation}\label{eq:Q0}
B_h(\uu v,q) = 0\qquad\forall (\bm v,q)\in\bm V_h\times\mathcal Q_0(\Omega). 
\end{equation}
Conversely, if, for given~$q\in\wt Q_h$, there holds that~$B_h(\uu v,q) = 0$ for any~$\bm v\in\bm V_h$, then it follows that~$q\in\mathcal Q_0(\Omega)$.
\end{lemma}

\begin{proof}
Given $(\uu v, q)\in \uu V_h\times \mathcal Q_0(\Omega)$. Since~$\mathcal Q_0(\Omega)$ is one-dimensional we may, without loss of generality, suppose that~$q\equiv 1$. Then, we have
\begin{equation*}
B_h(\uu{v},q) = -\sum_{K\in\T}\int_{K} \nabla_h\cdot\uu{v} \, \dx +
\int_{\F} \jmp{\uu{v}} \, \ds . 
\end{equation*}
By applying the Gauss-Green theorem on each element $K\in\T$, we obtain two expressions
which are identical, and, thus, cancel out:
\begin{equation*}
B_h(\uu{v},q) = -\sum_{K\in\T}\int_{\partial K} \uu{v}\cdot\uu n_K \,
\ds + \int_{\F} \jmp{\uu{v}} \, \ds  = 0.
\end{equation*}

Let now~$q\in\wt Q_h$, with~$q-\fint_{\Omega}q\,\dx\neq 0$, and~$B_h(\bm v,q)=0$ for all~$\bm v\in\bm V_h$. Then, the inf-sup condition~\eqref{eq:disc_infsup} implies that
\[
0<\gamma_B\le\sup_{\bm 0\neq\bm v\in\bm V_h}\frac{B_h\left(\bm v,q-\fint_{\Omega}q\,\dx\right)}{\|\bm v\|_h\NN{q-\fint_{\Omega}q\,\dx}_{L^2(\Omega)}}
=\sup_{\bm 0\neq\bm v\in\bm V_h}\frac{-B_h\left(\bm v,\fint_{\Omega}q\,\dx\right)}{\|\bm v\|_h\NN{q-\fint_{\Omega}q\,\dx}_{L^2(\Omega)}}.
\]
Applying~\eqref{eq:Q0} yields a contradiction, and, consequently, we deduce that $q-\fint_{\Omega}q\,\dx=0$. Thus, we have $q\equiv\fint_{\Omega}q\,\dx\in\mathcal{Q}_0$. This completes the proof.
\end{proof}

Now we can state the equivalence of the two formulations.
\begin{proposition}
The augmented DG discretization from~\eqref{eq:reform_mixedDG} has a
unique solution $(\wt{\uu u}_h, \wt{p}_h, 0)\in
\uu V_h\times \wt{Q}_h\times \IR$, and $(\wt{\uu u}_h, \wt{p}_h)$ is the solution of
the original DG formulation~\eqref{eq:mixedDG} with $\wt{p}_h\in Q_h$.
\end{proposition}

\begin{proof}
We proceed in three steps.

\subsubsection*{Step~1:} We first show that the new formulation enforces the pressure
$\wt{p}_h$ to have zero mean. To this end, we choose the test variable to be
$\wt{s}=1$. Thus, from the third equation of~\eqref{eq:reform_mixedDG}, we deduce that
\begin{equation*}
\fint_\Omega \wt{p}_h\,\dx = \wt{r}.
\end{equation*}
Furthermore, let us choose  $\wt{q}\equiv 1\in\mathcal{Q}_0$. From the second equation
of~\eqref{eq:reform_mixedDG}, and with the aid of~\eqref{eq:Q0}, we infer
that
\begin{equation*}
0 = (1-2\nu)\int_\Omega \wt{p}_h\,\dx - \Bigg(\fint_\Omega \wt{p}_h\,\dx\Bigg)
\Bigg(\int_\Omega 1\,\dx\Bigg) =
-2\nu\int_\Omega \wt{p}_h\,\dx,
\end{equation*}
i.e. $\fint_\Omega \wt{p}_h\,\dx = \wt{r} = 0$, since $\nu\neq 0$.

\subsubsection*{Step~2:} Next, we show that $(\wt{\uu u}_h, \wt p_h, \wt r):=(\uu u_h, p_h,
0)\in \uu V_h\times Q_h\times \IR$, where $(\uu u_h, p_h)$ is the solution
from~\eqref{eq:mixedDG}, solves~\eqref{eq:reform_mixedDG}. The first and last
equation in~\eqref{eq:reform_mixedDG} are clearly fulfilled with $(\wt{\uu u}_h,
\wt p_h, \wt r)=(\uu u_h, p_h,0)$. The second equation
in~\eqref{eq:reform_mixedDG} simplifies to
\begin{equation*}
-B_h(\uu u_h,\wt{q}) + C_{h}(p_h,\wt{q})=0,\qquad\forall\wt{q}\in\wt{Q}_h.
\end{equation*}
To show that this equation does indeed hold true, we notice that
\begin{align*}
 -B_h(\uu u_h,\wt{q}) + C_{h}(p_h,\wt{q}) &= - B_h\left(\uu u_h,\fint_\Omega
\wt{q}\,\dx\right) - B_h\left(\uu u_h,\wt{q}-\fint_\Omega \wt{q}\,\dx\right)\\
&\quad + C_{h}\left(p_h,\fint_\Omega \wt{q}\,\dx\right) + C_{h}\left(p_h,\wt{q} -
\fint_\Omega \wt{q}\,\dx\right).
\end{align*}
Due to~\eqref{eq:Q0}, the first term on the right-hand side is zero.
For the second term, since $\wt{q}-\fint_\Omega \wt{q}\in Q_h$, it holds
\begin{equation*}
B_h\left(\uu u_h,\wt{q}-\fint_\Omega \wt{q}\,\dx\right) = C_{h}\left(p_h,\wt{q} -
\fint_\Omega \wt{q}\,\dx\right),
\end{equation*}
simply by the second equation in~\eqref{eq:mixedDG}. Hence, we end up with
\begin{align*}
 -B_h(\uu u_h,\wt{q}) + C_{h}(p_h,\wt{q}) &= C_{h}\left(p_h,\fint_\Omega
\wt{q}\,\dx\right)\\
&= (1-2\nu)
\Bigg(\int_\Omega p_h\,\dx\Bigg)\Bigg(\fint_\Omega\wt{q}\,\dx\Bigg)=0,
\end{align*}
where we have used the fact that $p_h$ has zero mean. Thus, all three equations
from~\eqref{eq:reform_mixedDG} are fulfilled with $(\wt{\uu u}_h, \wt p_h, \wt
r)=(\uu u_h, p_h, 0)$ from above.

\subsubsection*{Step~3:} It remains to show that the solution
of~\eqref{eq:reform_mixedDG} is unique. To this end, we assume that there exist
two solutions $(\wt{\uu u}_{h1}, \wt{p}_{h1}, 0),(\wt{\uu u}_{h2}, \wt{p}_{h2},
0)\in \uu V_h\times Q_h\times \IR$. Using again~\eqref{eq:Q0} as well as the fact that the pressures have zero mean, the second equation
in~\eqref{eq:reform_mixedDG} implies that
\begin{align*}
0 &= - B_h\left(\wt{\uu u}_{h1} - \wt{\uu u}_{h2},\wt{q} - \fint_\Omega
\wt{q}\,\dx\right) - B_h\left(\wt{\uu u}_{h1} - \wt{\uu u}_{h2},\fint_\Omega
\wt{q}\,\dx\right)\\
&\quad + C_{h}\left(\wt{p}_{h1} - \wt{p}_{h2},\wt{q} - \fint_\Omega
\wt{q}\,\dx\right) + C_{h}\left(\wt{p}_{h1} - \wt{p}_{h2},\fint_\Omega
\wt{q}\,\dx\right)\\
&= - B_h\left(\wt{\uu u}_{h1} - \wt{\uu u}_{h2},\wt{q} - \fint_\Omega
\wt{q}\,\dx\right) + C_{h}\left(\wt{p}_{h1} - \wt{p}_{h2},\wt{q} -
\fint_\Omega \wt{q}\,\dx\right).
\end{align*}
Therefore, we get the following system for the difference $(\wt{\uu u}_{h1} - \wt{\uu u}_{h2}, \wt{p}_{h1} - \wt{p}_{h2})\in \uu V_h\times Q_h$ of the two
solutions: 
\begin{equation*}
\begin{split}
A_h(\wt{\uu u}_{h1} - \wt{\uu u}_{h2},\wt{\uu v})+B_h(\wt{\uu v},\wt{p}_{h1} -
\wt{p}_{h2})&=0, \\
-B_h\left(\wt{\uu u}_{h1} - \wt{\uu u}_{h2},\wt{q} - \fint_\Omega \wt{q}\,\dx\right)
+ C_{h}\left(\wt{p}_{h1} - \wt{p}_{h2},\wt{q} - \fint_\Omega \wt{q}\,\dx\right) &=0,
\end{split}
\end{equation*}
for all $(\wt{\uu v}, \wt{q} - \fint_\Omega \wt{q}\,\dx)\in \uu V_h\times Q_h$. This, in turn, is just the mixed formulation~\eqref{eq:mixedDG} with the unique zero solution.
\end{proof}

\subsection{Inf-sup constant of the form $B_h$}
We will now discuss how to compute the inf-sup constant~$\gamma_B$ from~\eqref{eq:disc_infsup} numerically. To do so, we proceed along the lines of~\cite[Chapter~II.3.2]{BrezziFortin91}. Let us choose two sets of basis functions $\{\uu \phi_i\}_{i=1}^M\subset\bm V_h$ and $\{\psi_j\}_{j=1}^N\subset\wt Q_h$, with~$M=\dim(\bm V_h)$ and~$N=\dim(\wt Q_h)$. For any~$\bm v\in V_h$ and~$q\in\wt Q_h$, we store the associated coefficients in two vectors~$\uuu v := (v_1,\dots,v_M)^\top$ and $\uuu q :=(q_1,\dots,q_N)^\top$, respectively. Counting degrees of freedom in each element, we remark that
\begin{equation}\label{eq:MN}
M=\frac{3(k+1)^3}{k^3}N>3N.
\end{equation}
Furthermore, we define the matrix~$\uuu B \in \IR^{M\times N}$ by
\[
B_{ij}:=B_h(\uu\phi_i,\psi_j),\qquad 1\le i\le M,\, 1\le j\le N.
\]
Due to Lemma~\ref{lem:kernelB}, we notice that~$q\in\mathcal{Q}_0$ if and only if the associated coefficient vector~$\uuu q$ satisfies $\uuu q\in\ker(\uuu B)$. Moreover, by virtue of~\eqref{eq:quot}, we conclude that $q\in Q_h$ if and only if $\uuu q\in\mathbb{R}^N/\ker(\uuu B)$. Moreover, since~$\dim(\mathcal{Q}_0)=1$, it follows that $\dim(\ker(\uuu B))=1$, and, in view of~\eqref{eq:MN},
\begin{equation}\label{eq:rank}
\rank(\uuu B)=\rank(\uuu B^\top)=N-1.
\end{equation}

Let us further introduce the symmetric positive definite matrices~$\uuu D\in \IR^{M\times M}$ and~$\uuu E\in\IR^{N\times N}$ corresponding to the norms~$\|\cdot\|_h$ and~$\|\cdot\|_{L^2(\Omega)}$, respectively, through
\[
\|\bm v\|_h^2=\uuu v^\top\uuu D\uuu v,\qquad
\|q\|^2_{L^2(\Omega)}=\uuu q^\top\uuu E\uuu q.
\]
Taking into account our considerations above, we infer that
\begin{align*}
\gamma_B&=  \inf_{\uuu 0\neq
\uuu q \in \mathbb{R}^N/\ker(\uuu B)}\sup_{\uuu 0\neq\uuu v\in\mathbb{R}^M}\frac{\uuu v^\top \uuu B \,\uuu
q}{(\uuu v^\top \uuu D\, \uuu v)^{\nicefrac12}(\uuu q^\top \uuu E\, \uuu
q)^{\nicefrac12}}\\
&=  \inf_{\uuu 0\neq
\uuu q \in \mathbb{R}^N/\ker(\uuu B)} \sup_{\uuu 0\neq\uuu v\in\mathbb{R}^M/\ker(\uuu B^\top)}\frac{\uuu v^\top \uuu B \,\uuu
q}{(\uuu v^\top \uuu D\, \uuu v)^{\nicefrac12}(\uuu q^\top \uuu E\, \uuu
q)^{\nicefrac12}}.
\end{align*}
To proceed, we define
\begin{equation}\label{eq:BSmatrix}
\wt{\uuu B} := \uuu D^{-\nicefrac12}\uuu B \uuu E^{-\nicefrac12}.
\end{equation}
Let the singular value decomposition (SVD) of $\wt{\uuu B}$ be given by
\[
\wt{\uuu B} =: \wt{\uuu V} \uuu\Sigma \wt{\uuu Q}^\top,
\]
with the singular values
$\sigma_1\ge\dots\ge\sigma_{N-1}>\sigma_{N}=0$, cf.~\eqref{eq:rank}, being
contained on the diagonal of~$\uuu\Sigma$, and orthogonal matrices $\wt{\uuu
V}$ and $\wt{\uuu Q}$ with columns $\wt{\uuu v}_1,\dots,\wt{\uuu v}_M$ and $\wt{\uuu
q}_1,\dots,\wt{\uuu q}_N$, respectively. In addition, we set
\begin{equation}\label{eq:Relation}
\wh{\uuu v}_i := \uuu D^{-\nicefrac12}\,\wt{\uuu v}_i \quad \forall\,
i\in\{1,\dots,M\},\qquad
\wh{\uuu q}_i := \uuu E^{-\nicefrac12}\,\wt{\uuu q}_i \quad \forall
\,i\in\{1,\dots,N\}.
\end{equation}
For $i\in\{1,\dots,N\}$ we then conclude
\begin{equation*}
\uuu B\,\wh{\uuu q}_i = \uuu D^{\nicefrac12}\wt{\uuu B}\uuu E^{\nicefrac12}\,\wh{\uuu q}_i
= \uuu D^{\nicefrac12}\wt{\uuu V} \uuu\Sigma \wt{\uuu Q}^\top\uuu E^{\nicefrac12}\,\wh{\uuu
q}_i
= \uuu D^{\nicefrac12}\wt{\uuu V} \uuu\Sigma \wt{\uuu Q}^\top\,\wt{\uuu q}_i
= \uuu D^{\nicefrac12}\wt{\uuu V} \uuu\Sigma \,\uuu e_i,
\end{equation*}
where $\uuu e_i$ is the $i$th standard unit vector in~$\IR^N$. Hence, we
have
\begin{equation}\label{eq:Bq}
\uuu B\,\wh{\uuu q}_i
= \sigma_i \uuu D^{\nicefrac12}\wt{\uuu V} \,\wt{\uuu e}_i
= \sigma_i \uuu D^{\nicefrac12}\,\wt{\uuu v}_i
= \sigma_i \uuu D\,\wh{\uuu v}_i,
\end{equation}
where $\wt{\uuu e}_i$ is the $i$th standard unit vector in~$\IR^M$.
Involving~\eqref{eq:Relation}, we deduce that
\begin{equation}\label{eq:Orthogonality}
\begin{split}
\wh{\uuu v}^\top_i\uuu D\,\wh{\uuu v}_j&=\delta_{ij} \qquad \forall\,
i,j\in\{1,\dots,M\},\\
\wh{\uuu q}^\top_i\uuu E\,\wh{\uuu q}_j&=\delta_{ij} \qquad \forall\,
i,j\in\{1,\dots,N\}.
\end{split}
\end{equation}
Given linear combinations
\begin{equation}\label{eq:LinComb}
\uuu v =: \sum_{i=1}^{N-1} \alpha_i \,\wh{\uuu v}_i \in \mathbb{R}^M/\ker(
\uuu B^\top),\qquad
\uuu q =: \sum_{j=1}^{N-1} \beta_j \,\wh{\uuu q}_j \in \mathbb{R}^N/\ker(
\uuu B),
\end{equation}
and using~\eqref{eq:Bq} and~\eqref{eq:Orthogonality}, we obtain
\[
\uuu v^\top \uuu B \,\uuu q = \sum_{i,j=1}^{N-1} \alpha_i \beta_j \,\wh{\uuu v}_i^\top \uuu B
\,\wh{\uuu
q}_j
= \sum_{i,j=1}^{N-1} \sigma_j \alpha_i \beta_j \,\wh{\uuu v}^\top_i \uuu D\,\wh{\uuu v}_j
= \sum_{i=1}^{N-1} \sigma_i \alpha_i \beta_i.
\]
Moreover, employing~\eqref{eq:Orthogonality} and~\eqref{eq:LinComb}, the norms are represented by
\begin{align*}
\N{\uu v}_h&=\big(\uuu v^\top \uuu D \,\uuu v\big)^{\nicefrac12} =
\Bigg(\sum_{i=1}^{N-1}
\alpha_i^2\Bigg)^{\nicefrac12},&
\N{q}_{L^2(\Omega)}&=\big(\uuu q^\top \uuu E \,\uuu q\big)^{\nicefrac12} =
\Bigg(\sum_{i=1}^{N-1}
\beta_i^2\Bigg)^{\nicefrac12}.
\end{align*}

We will now evaluate the inf-sup constant
\begin{equation}\label{eq:gamma_B}
\gamma_B = \inf_{
\uu\beta\neq\uu0} \sup_{\uu\alpha\neq\uu0}\frac{\sum_{i=1}^{N-1} \sigma_i \alpha_i
\beta_i}{\Big(\sum_{i=1}^{N-1} \alpha_i^2\Big)^{\nicefrac12}\Big(\sum_{i=1}^{N-1}
\beta_i^2\Big)^{\nicefrac12}},
\end{equation}
with $\uu\alpha := (\alpha_1,\dots,\alpha_{N-1})$ and $\uu\beta :=
(\beta_1,\dots,\beta_{N-1})$.
Without loss of generality, we may suppose that $\N{\uu v}_h=\big(\sum_{i=1}^{N-1}
\alpha_i^2\big)^{\nicefrac12}=1$, and $\N{q}_{L^2(\Omega)}=\big(\sum_{i=1}^{N-1}
\beta_i^2\big)^{\nicefrac12}=1$. Then, \eqref{eq:gamma_B} simplifies to
$\gamma_B = \inf_{\uu\beta} \sup_{\uu\alpha}\sum_{i=1}^{N-1} \sigma_i
\alpha_i\beta_i$. 
By applying the Cauchy-Schwarz inequality, that is,
\begin{equation*}
\sum_{i=1}^{N-1} \sigma_i \alpha_i\beta_i\le \left(\sum_{i=1}^{N-1}
\alpha_i^2\right)^{\nicefrac12} \left(\sum_{i=1}^{N-1}
\sigma_i^2\beta_i^2\right)^{\nicefrac12}
=\left(\sum_{i=1}^{N-1} \sigma_i^2\beta_i^2\right)^{\nicefrac12},
\end{equation*}
we observe that the supremum is attained for
$\alpha_i = \big(\sum_{i=1}^{N-1}
\sigma_i^2\beta_i^2\big)^{-\nicefrac12}\sigma_i\beta_i$. This leads to
\begin{equation*}
\gamma_B = \inf_{\beta^\top\beta=1} \left(\sum_{i=1}^{N-1}
\sigma_i^2\beta_i^2\right)^{\nicefrac12}=\sigma_{N-1}.
\end{equation*}

\begin{proposition}
The inf-sup constant $\gamma_B$ from~\eqref{eq:disc_infsup} is given by
the smallest positive singular value, $\sigma_{N-1}$, of the matrix $\wt{\uuu B}$
from~\eqref{eq:BSmatrix}.
\end{proposition}

\subsection{Inf-sup constant of the form $a_h$}

In order to compute the inf-sup constant from~\eqref{eq:disc_infsupa}, we proceed analogously as in the previous section. To this end, we choose a basis~$\{\uu \chi_i\}_{i=1}^{M+N}$ of $\uu V_h\times\wt{Q}_h$, and define the system matrix~$\uuu M \in \IR^{(M+N)\times (M+N)}$ by
\begin{equation*}
M_{ij}:=a_h(\uu\chi_j,\uu\chi_i), \qquad 1\le i,j\le M+N,
\end{equation*}
with~$a_h$ from~\eqref{eq:formah}. For brevity, we
consider only the limit case $\nu=\nicefrac12$. Due to Lemma~\ref{lem:kernelB} and the coercivity of the form~$A_h$ from~\eqref{eq:mixedDG}, we conclude that $\uuu M$ has a one-dimensional kernel. Denoting by $\uuu D\in\IR^{(M+N)\times (M+N)}$ the symmetric positive matrix defining the~$\tnorm{\cdot}_{\DG}$-norm through
\[
\tnorm{(\bm u,p)}^2_{\DG}=\uuu v^\top\uuu D\uuu v,
\]
where~$\uuu v\in\mathbb{R}^{M+N}$ is the coefficient vector of a given pair~$(\bm u,p)\in\bm V_h\times\wt Q_h$ with respect to the basis~$\{\bm\chi_i\}_{i=1}^{M+N}$, we let
\begin{equation}\label{eq:M0SMatrix}
\wt{\uuu M} := \uuu D^{-\nicefrac12}\uuu M\uuu D^{-\nicefrac12}.
\end{equation}
Given the SVD of $\wt{\uuu M}$ by $\wt{\uuu M} =: \wt{\uuu X} \uuu\Sigma \wt{\uuu Y}^\top$, with the singular values
$\sigma_1\ge\dots\ge\sigma_{M+N-1}>\sigma_{M+N}=0$ being
contained on the diagonal of~$\uuu\Sigma$, and orthogonal matrices
$\wt{\uuu X}$ and $\wt{\uuu Y}$, we infer the following result.

\begin{proposition}
In the incompressible case~$\nu=\nicefrac12$, the inf-sup constant from~\eqref{eq:disc_infsupa} satisfies $\gamma_a=\sigma_{M+N-1}$, where~$\sigma_{M+N-1}$ is the smallest positive singular value of the matrix $\wt{\uuu M}$ from~\eqref{eq:M0SMatrix}.
\end{proposition}


\subsection{Numerical computation of the inf-sup constants}

We shall now investigate the inf-sup constants~$\gamma_B$ and~$\gamma_a$ from~\eqref{eq:disc_infsup} and~\eqref{eq:disc_infsupa}, respectively, by means of a number of numerical experiments. In particular, we will investigate the dependence on the approximation degree $k$ and on the Poisson ratio~$\nu$.  
In the sequel, we choose $\theta$ from~\eqref{eq:DGforms} to be 1 (i.e. we use the symmetric interior
penalty DG method), the mesh grading factor from
Section~\ref{sc:meshes} as $\sigma = \nicefrac12$, and the
penalty parameter from~\eqref{eq:c} is set to $\gamma = 10$. As shape functions we use tensorized Lagrange polynomials in the Gauss quadrature points. 
All our computations are performed with the finite element library \textsf{deal.II}; see, e.g.,~\cite{BangerthHartmannKanschat2007,DealIIReference}.

\subsubsection{Inf-sup constant~$\gamma_B$}\label{sec:CompInfSupA}
We consider the canonical edge, corner, and corner-edge patch meshes presented in Section~\ref{sc:meshes}, see Figure~\ref{fig:Meshes}, with the modification that, in case of the corner-edge mesh, we refine one corner and \emph{all} adjacent neighbouring edges. Furthermore, we study the situation of geometrically refined meshes on a Fichera domain given by $(-1,1)^3\setminus[0,1)^3$, where we simultaneously refine the reentrant corner and all three adjacent edges.

Let us first fix the approximation degree $k$, and refine the meshes by increasing the number of refinement levels~$\ell$ step by step. For different approximation degrees the results are depicted in Figure~\ref{fig:InfSupB}. We clearly observe that the values of~$\gamma_B$ stabilize after some initial refinement steps, thereby underlining the robustness of the inf-sup constant of~$B_h$ with respect to the anisotropic geometric refinements. The asymptotic values are visualized in Figure~\ref{fig:InfSupBEndpoints}; it is observed that there is a mild $k$-dependence of the inf-sup constant~$\gamma_B$, however, our results indicate that, for the given examples, the dependence is considerably more optimistic than the theoretical bound $k^{-\nicefrac32}$ proved in~\cite{Schoetzau:2004:MHP}, cf.~Remark~\ref{rm:rho}.

\begin{figure}
\begin{center}
\subfigure{\includegraphics[width=0.49\textwidth]{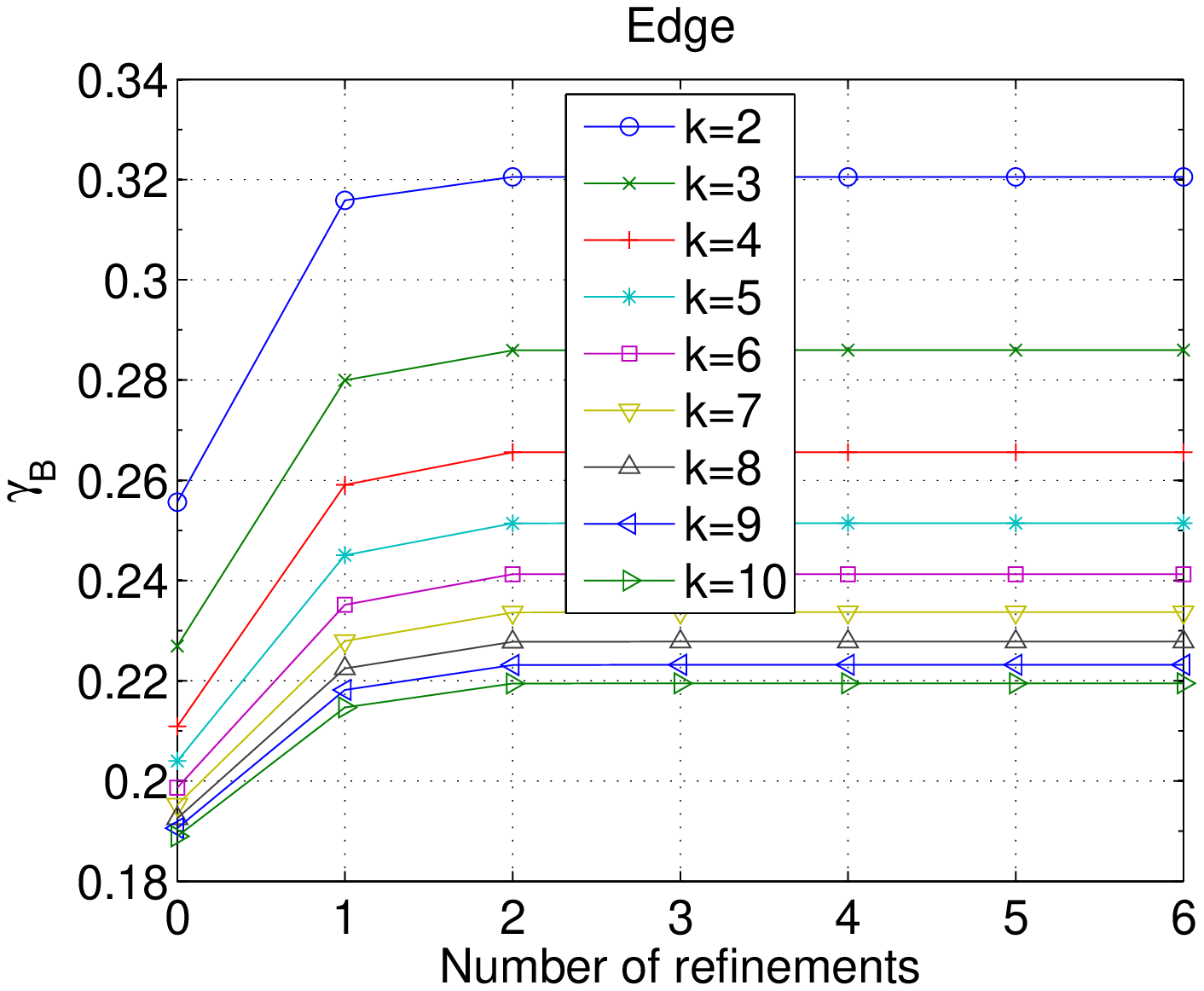}}
\subfigure{\includegraphics[width=0.49\textwidth]{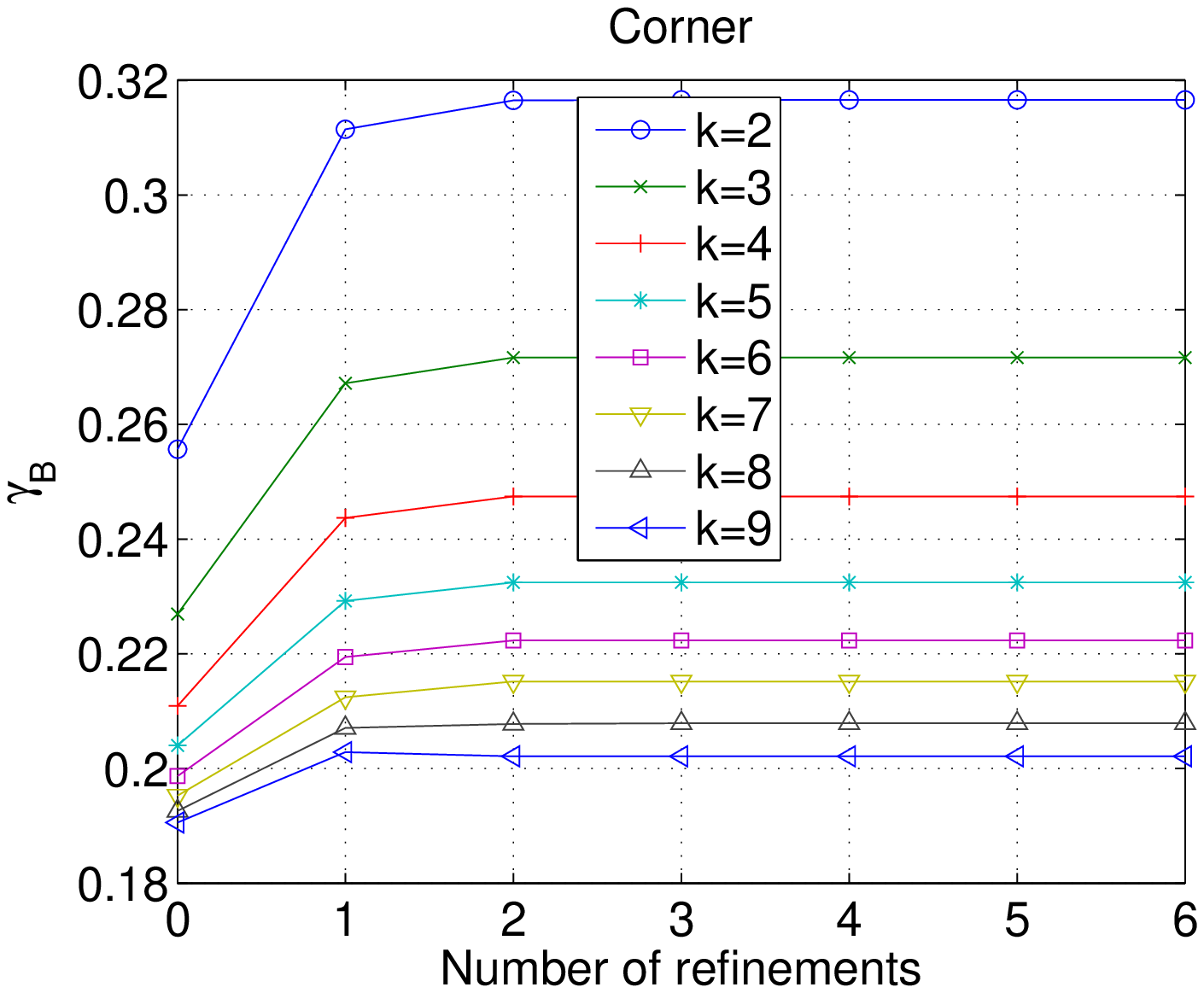}}
\subfigure{\includegraphics[width=0.49\textwidth]{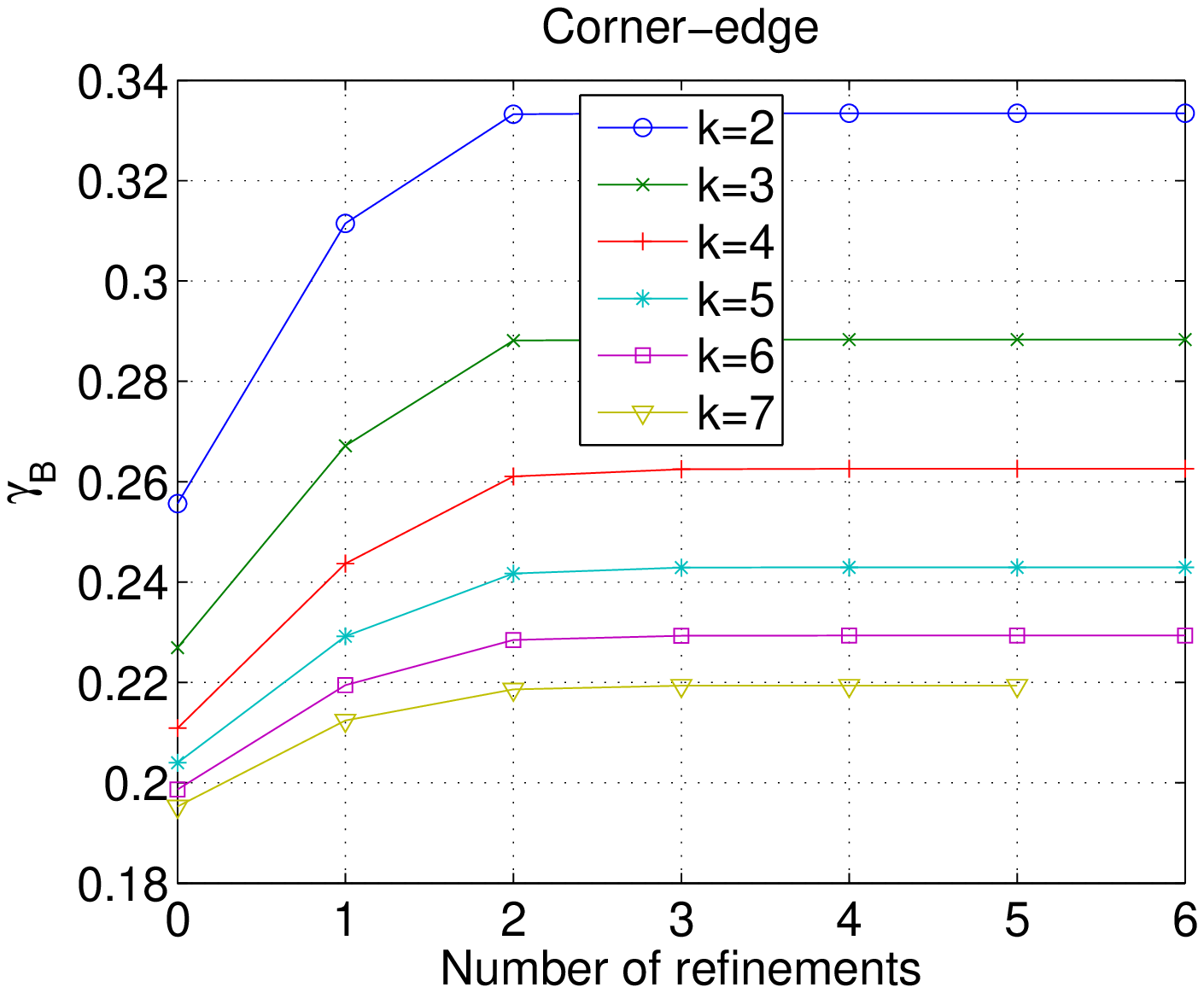}}
\subfigure{\includegraphics[width=0.49\textwidth]{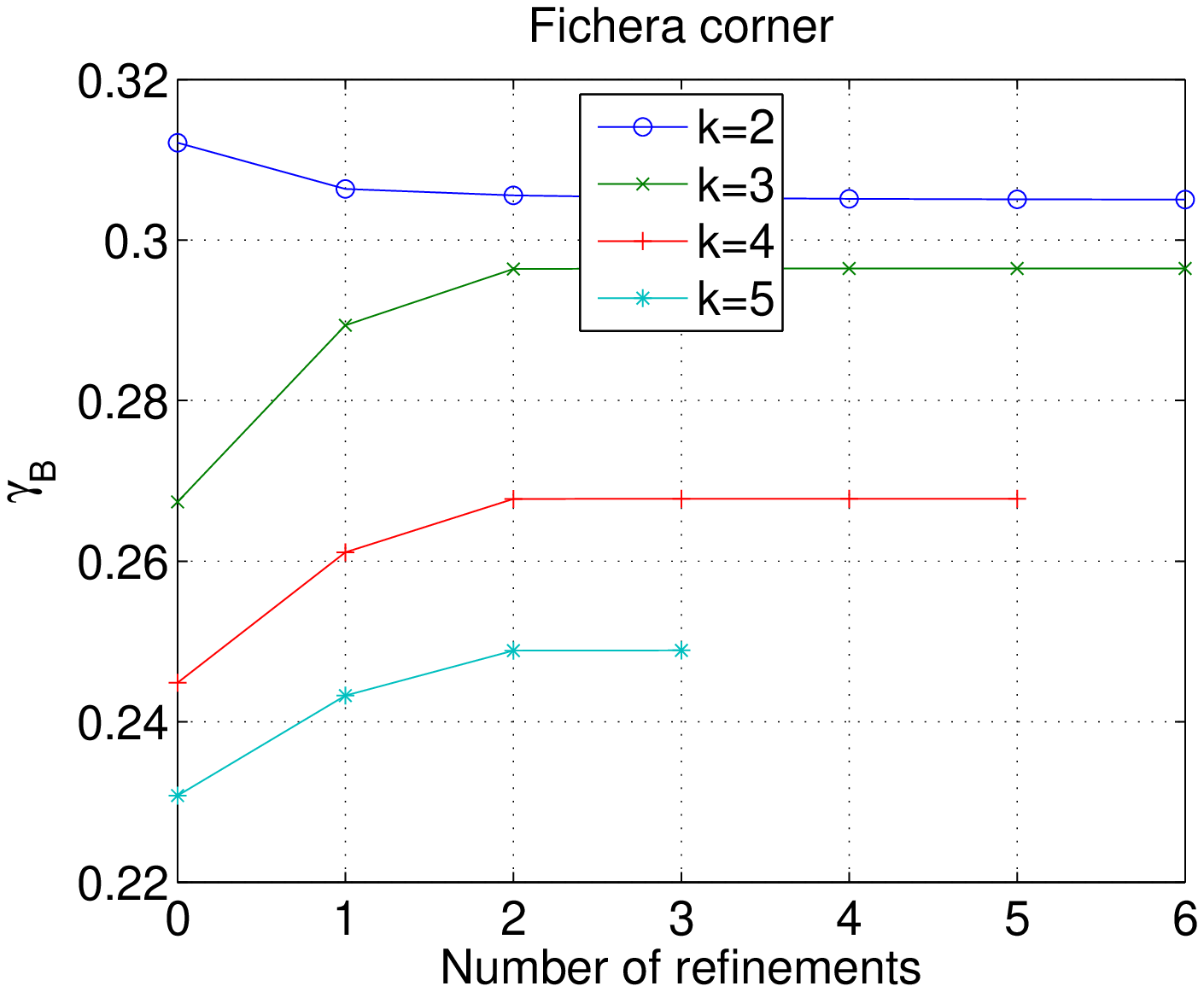}}
\end{center}
\caption{Inf-sup constant~$\gamma_B$ of the form $B_h$ in case of geometrically refined
edge, corner, corner-edge patches, as well as for a Fichera corner
refinement, for different approximation degrees~$k$.}
\label{fig:InfSupB}
\end{figure}

\begin{figure}
\begin{center}
\subfigure{\includegraphics[width=0.45\textwidth]{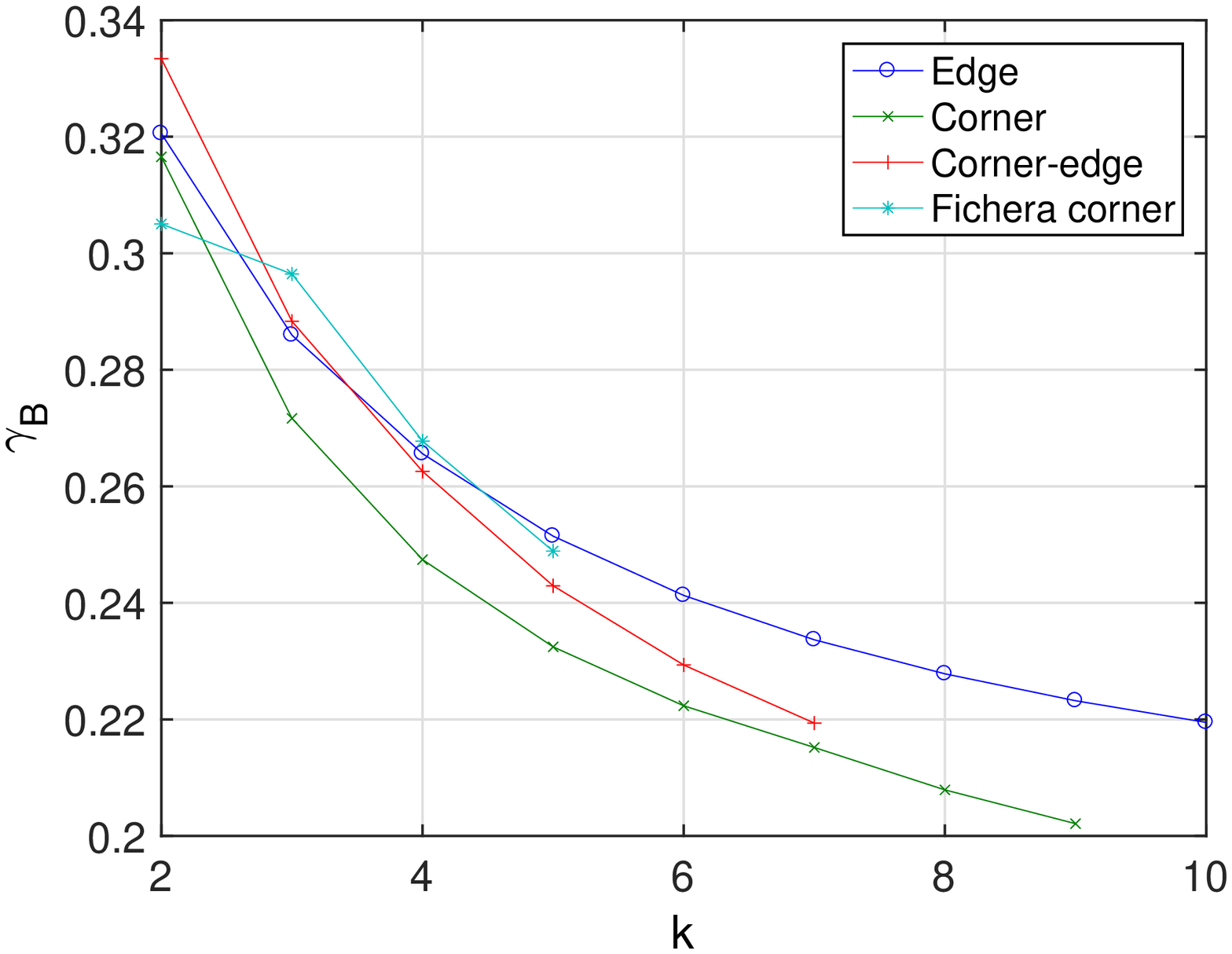}}\hspace{2ex}
\subfigure{
\includegraphics[width=0.45\textwidth]{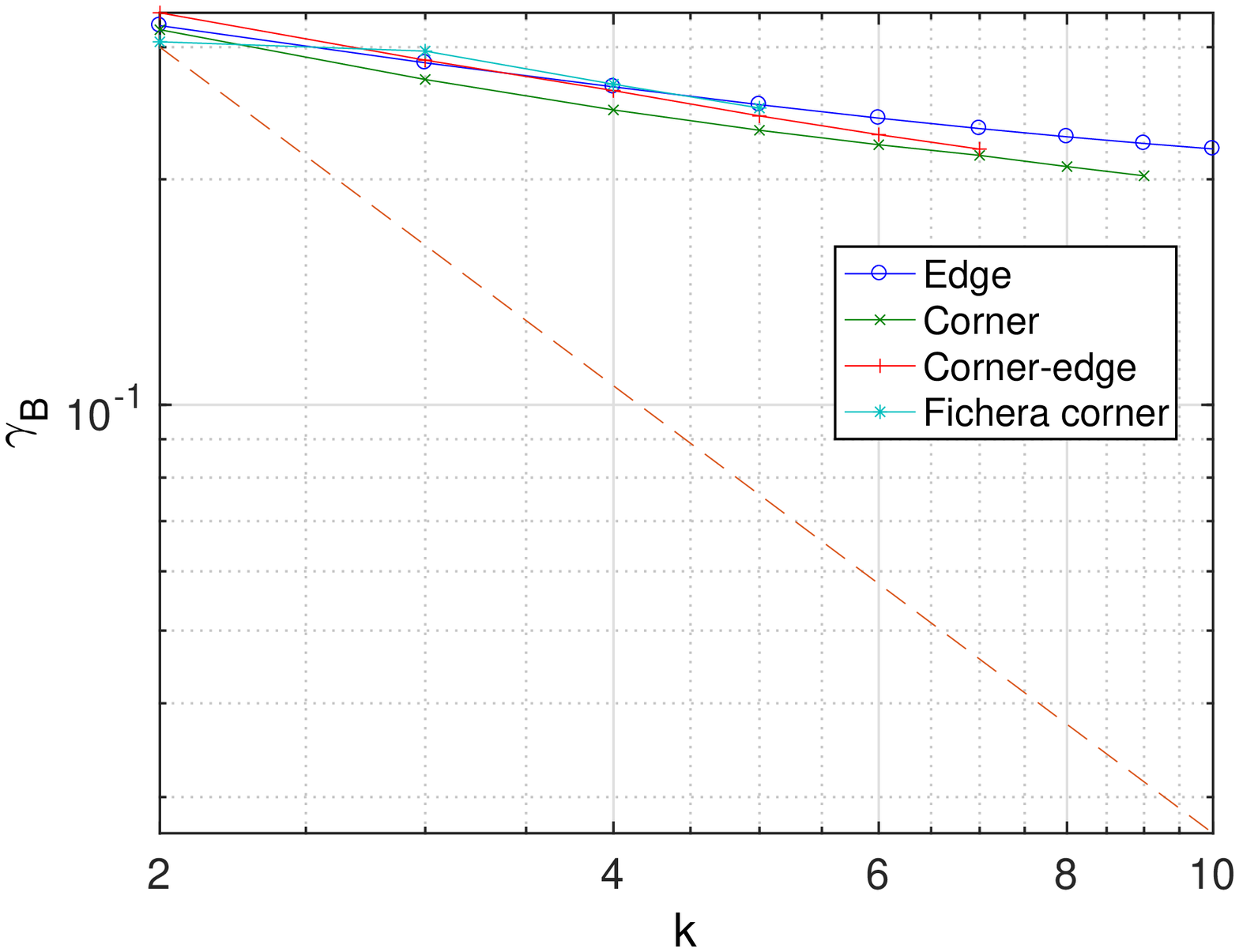}}
\end{center}
\caption{Stabilized values from Figure~\ref{fig:InfSupB} (on the right
with logarithmic scaling; the dashed line shows a slope of~$-\nicefrac32$).}
\label{fig:InfSupBEndpoints}
\end{figure}

\subsubsection{Inf-sup constant of the form ${a}_h$}
Let us turn to the behavior of the inf-sup constant $\gamma_a$
from~\eqref{eq:disc_infsupa} with respect to $k$, with $\nu=\nicefrac12$. From Theorem~\ref{Thm:disc_infsupa} recall the theoretical dependence $\gamma_a\gtrsim\max\{k^{-2\rho},1-2\nu\}$; this result holds true with a theoretical value of $\rho=\nicefrac{3}{2}$, cf.~Remark~\ref{rm:rho}. Since we set $\nu=\nicefrac12$, we deduce $\gamma_a\gtrsim k^{-3}$.

We focus on the canonical geometric edge, corner, and corner-edge refinements from Section~\ref{sc:meshes}. As before, we first fix the approximation degree~$k$,
and refine the meshes step by step in order to monitor the inf-sup constant; see
Figure~\ref{fig:InfSupA} for the resulting plots with different
approximation degrees. Again, we display the stabilized values of~$\gamma_a$ for increasing~$k$ in Figure~\ref{fig:InfSupAEndpoints}. The results are qualitatively similar to the inf-sup constant~$\gamma_B$ discussed earlier. There is a $k$-dependence of the inf-sup constant $\gamma_a$ which is again much
weaker than $k^{-3}$. Furthermore, our results show that the
inf-sup constant $\gamma_a$ does not deteriorate in the critical limit $\nu=\nicefrac12$ as shown in Theorem~\ref{Thm:disc_infsupa}.

\begin{figure}
\begin{center}
\subfigure{\includegraphics[width=0.49\textwidth]{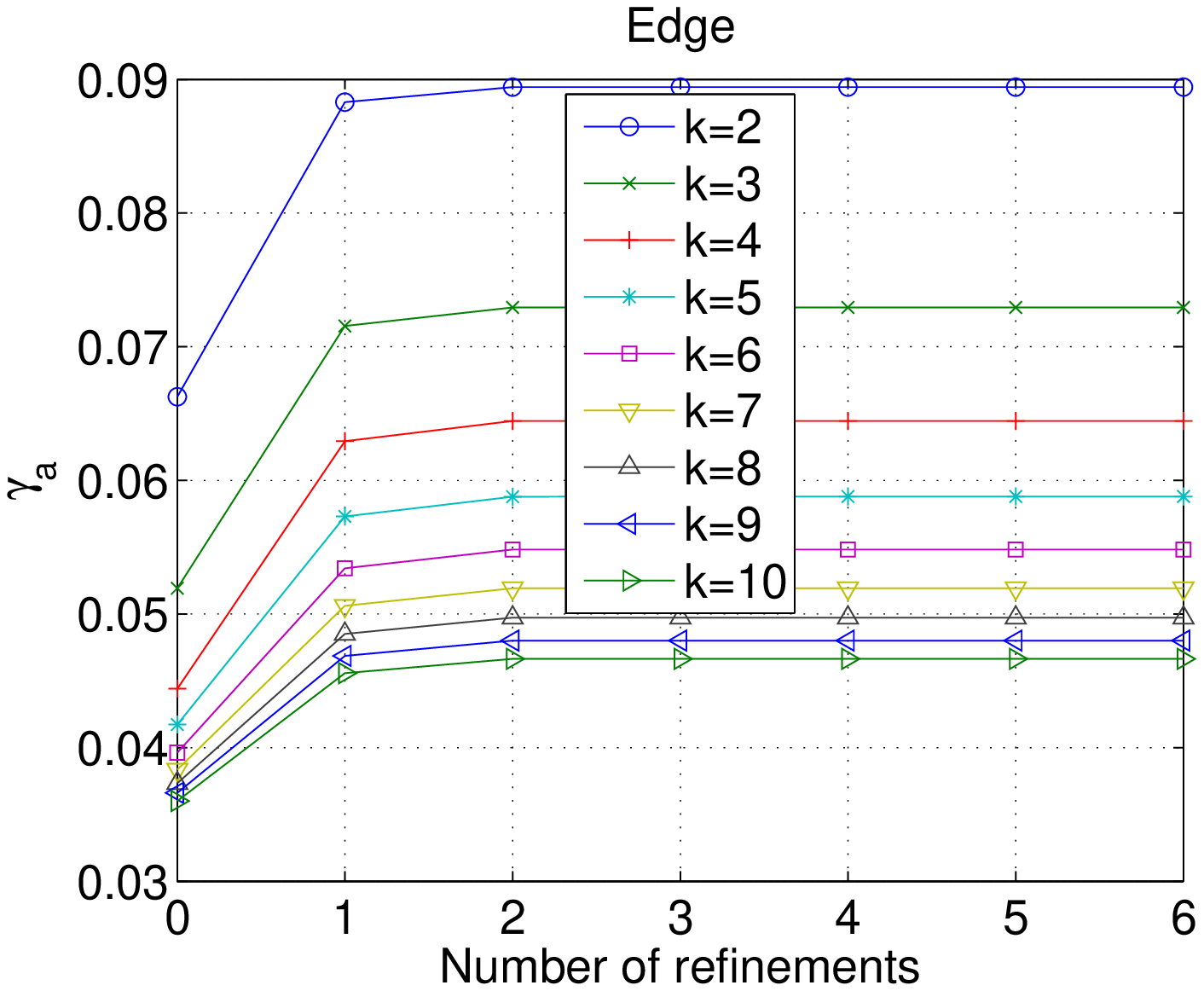}}
\subfigure{\includegraphics[width=0.49\textwidth]{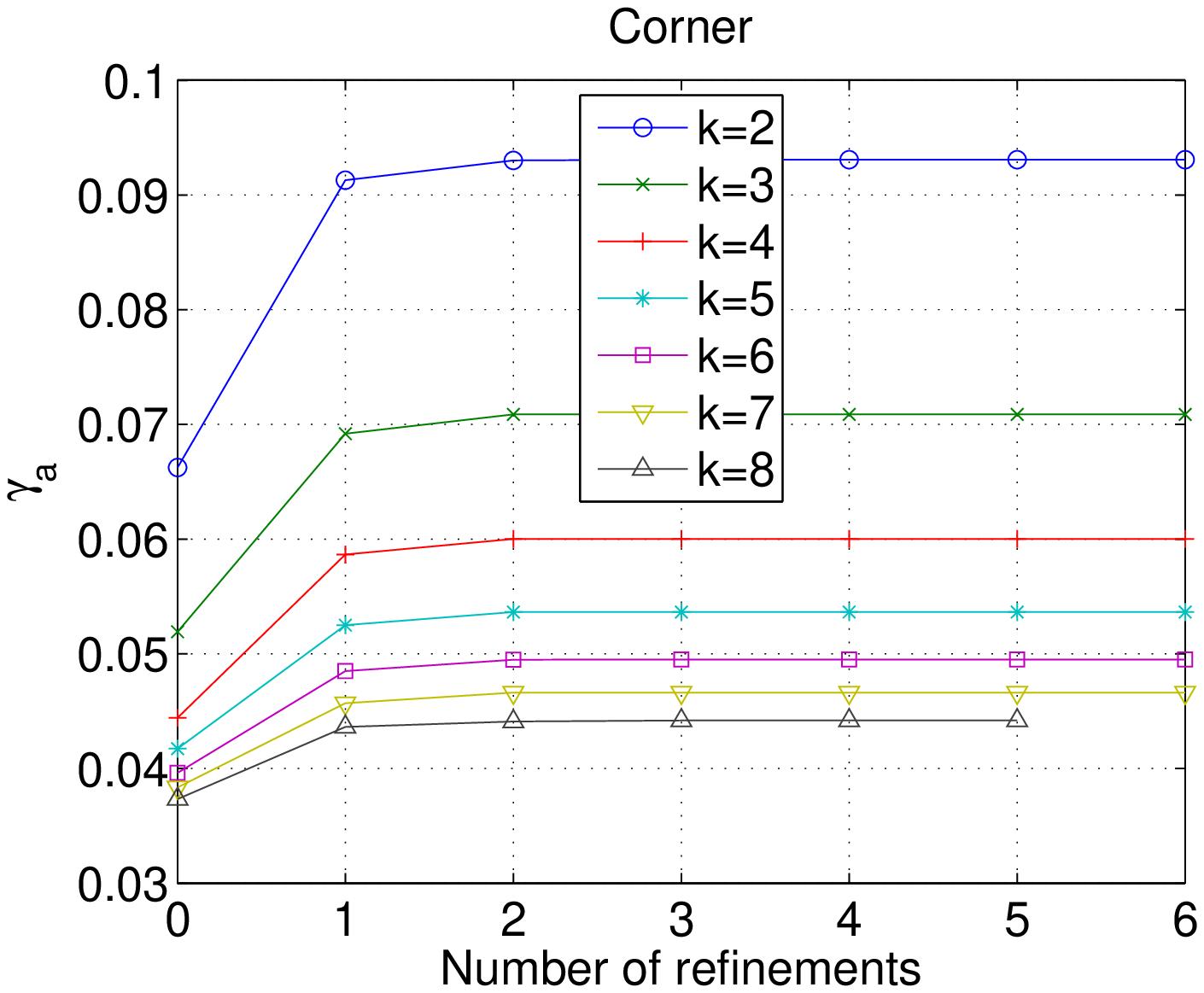}}
\subfigure{\includegraphics[width=0.49\textwidth]{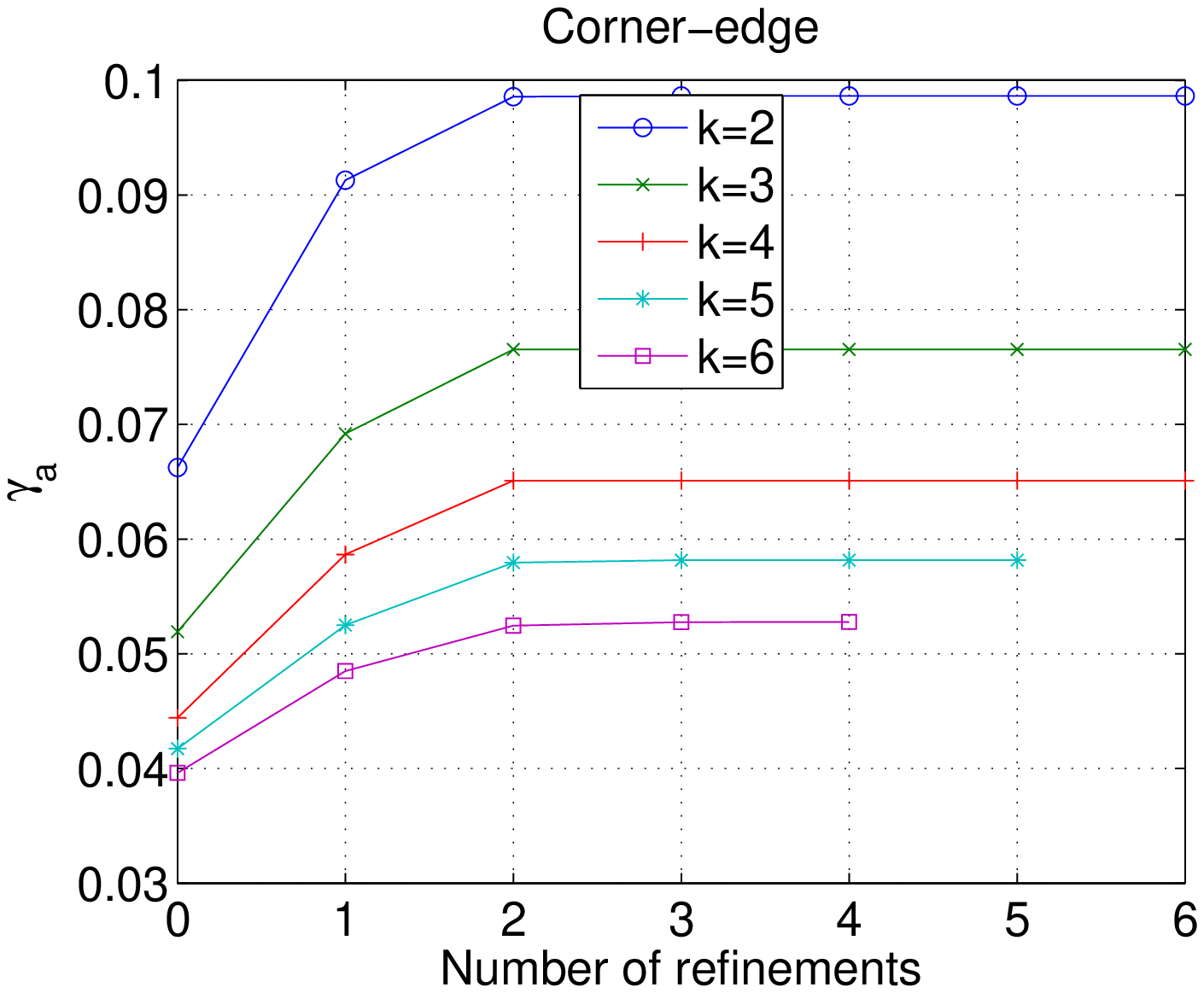}}
\end{center}
\caption{Inf-sup constant $\gamma_a$ of the form $a_h$  in case of geometric
edge, corner, and corner-edge refinements, with different approximation degrees $k$ and~$\nu=\nicefrac12$.}
\label{fig:InfSupA}
\end{figure}

\begin{figure}
\begin{center}
\subfigure{\includegraphics[width=0.49\textwidth]{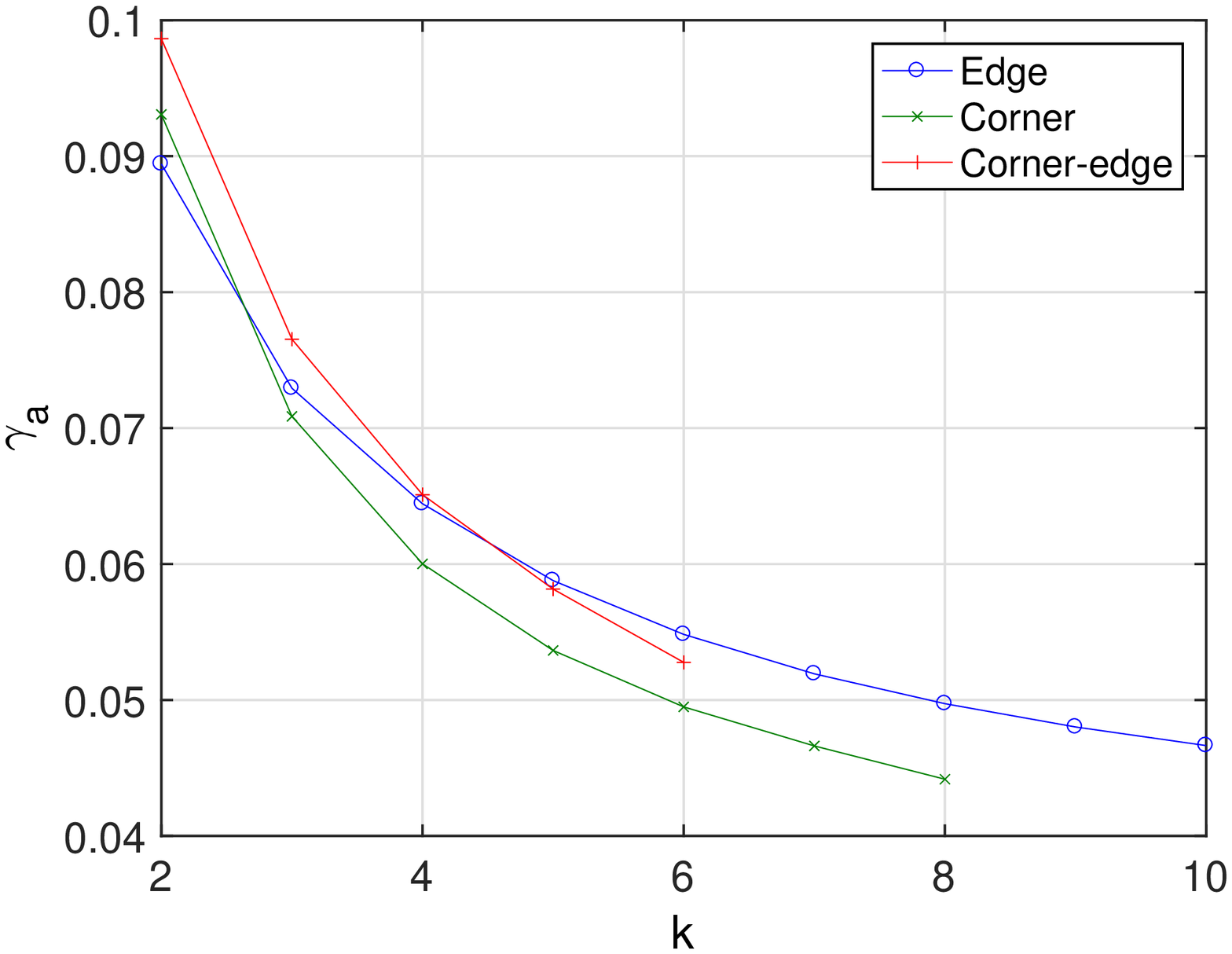}}\hfill
\subfigure{\includegraphics[width=0.49\textwidth]{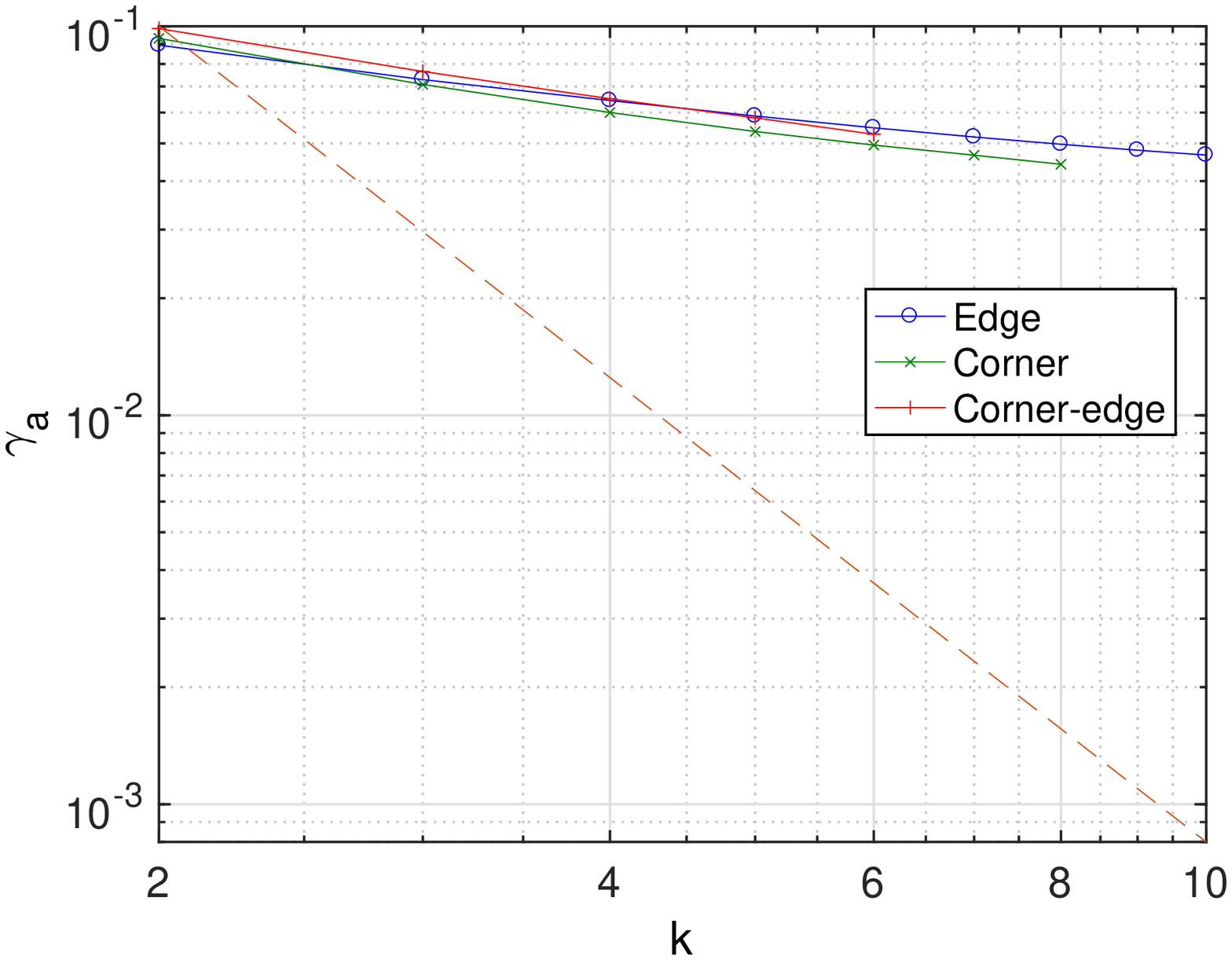}}
\end{center}
\caption{Stabilized values from Figure~\ref{fig:InfSupA} (on the right
with logarithmic scaling; the dashed line shows a slope of~$-3$).}
\label{fig:InfSupAEndpoints}
\end{figure}

\section{Exponential convergence on geometric meshes}\label{sec:CompResults}

In this section we turn to the exponential convergence of the spectral mixed DG method~\eqref{eq:mixedDG}. Inspired by the regularity theory from~\cite{DaugeCostabelNicaise} for analytic data (cf., in particular, Theorem~6.9), we
suppose that the solution~$(\bm u, p)$ of~\eqref{eq:problem1}--\eqref{eq:problem3}
belongs to~$A_{-1-\uu\beta}(\Omega)^{3}\times A_{-\uu\beta}(\Omega)$, where, for a weight vector~$\bm\gamma\in\IR^{\abs{\mathcal C}+\abs{\mathcal E}}$, we consider the countably normed space of piecewise analytic functions
\[
A_{\uu\gamma}(\Omega):=\bigg\{ v\in\bigcap_{m\ge 1}M^m_{\uu\gamma}(\Omega)
:\N{v}_{M^m_{\uu\gamma}(\Omega)}\le C^{m+1}_vm!\ \forall\, m\in\IN\bigg\},
\]
with a constant $C_v>0$ depending on the function $v$. Under this assumption, referring to~\cite[Theorem~6.2]{wihlerWirz}, it can be shown that the DG approximation~$(\bm u_h,p_h)$ from~\eqref{eq:mixedDG} converges at an exponential rate. To quantify this fact, let~$\nu\in(0,\nicefrac12]$, and consider a sequence of geometric edge meshes~$\T$ as in Section~\ref{sc:meshes}. Moreover, choose a uniform polynomial degree $k\ge2$ that is proportional to the number of layers $\ell\ge1$ in~$\T$. Then, the DG approximation $(\uu u_h,p_h)$ from~\eqref{eq:mixedDG} satisfies the error bound
\begin{equation}\label{eq:ExponentialConvergence}
\tnorm{(\uu u -\uu u_h,p-p_h)}_{DG}^2\lesssim\exp(-b\sqrt[5]{N}),
\end{equation}
where $N:=\dim(\uu V_h\times Q_h)$ denotes the number of degrees of
freedom.

\subsection{Exponential convergence on canonical patches}
In our numerical examples, we use manufactured solutions, which feature typical singularities close to~$\S$ in the displacement~$\bm u=(u_1,u_2,u_3)$, and then test the spectral DGFEM~\eqref{eq:mixedDG} with the resulting right-hand side force functions~$\bm f$ in~\eqref{eq:problem1}. The domain $\Omega$ is chosen to be the unit cube. To cover all possible cases, we consider one solution with an edge, one with a corner, and another one with a corner-edge singularity:
\begin{itemize}
\item Displacement with an anisotropic edge singularity along the $z$-axis:
\[
u_1=u_2=0,\qquad u_3=(x^2+y^2)^{\nicefrac{1}{4}}z(1-z).
\]
\item Displacement with an isotropic corner singularity at the origin:
\[
u_1=u_2=0,\qquad
u_3=(x^2+y^2+z^2)^{\nicefrac{1}{6}}z(1-z).
\]
\item Displacement which combines the above edge and corner
singularities:
\[
u_1=u_2=0,\qquad
u_3=(x^2+y^2+z^2)^{\nicefrac{1}{6}}(x^2+y^2)^{\nicefrac{1}{4}}z(1-z).
\]
\end{itemize}
The corresponding pressures $p$ for these displacements are then given
via~\eqref{eq:problem2}:
\[
p=-\frac{1}{1-2\nu}\nabla\cdot\uu u, \qquad \nu\neq\nicefrac{1}{2}.
\]
Incidentally, in order to avoid too complicated right-hand sides $\uu f$, we do not
enforce the displacements $\uu u$ to vanish on the whole boundary $\partial\Omega$. More precisely, they are chosen such that $\uu u\cdot\uu
n\big\vert_{\partial\Omega}=0$, i.e. the normal component of the displacement field
is zero. Then, by the Gauss-Green theorem, notice the mean value property of the pressure,
\[
\int_\Omega p\, \dx=-\frac{1}{1-2\nu}\int_\Omega\nabla\!\cdot\!\uu u\,
\dx=-\frac{1}{1-2\nu}\int_{\partial\Omega}\uu u\!\cdot\!\uu n\, \ds=0.
\]
The nonzero
Dirichlet boundary conditions are accounted for by means of a standard flux
term which is added to the right-hand side of the DGFEM formulation~
\eqref{eq:disccompact}.

For our test examples, we use $\nu\in\left\{\nicefrac18,\nicefrac12,\nicefrac38\right\}$. In order to solve the resulting linear systems we employ the \textsf{GMRES} method in combination with the \textsf{SparseILU} preconditioner implemented in \textsf{deal.II}. The iterations terminate as soon as the Euclidean norm of the (unpreconditioned) residual becomes smaller or equal to $10^{-12}$. The initial meshes consist of a single element, and an approximation degree $k=1$. In the following, we refine successively the meshes towards the singularities, and simultaneously increase the approximation degree by one in each refinement step such that
$k\sim \ell$, where $\ell$ is the number of layers. Since the singularities (and thereby their location) in the examples above are known explicitly, we only refine the corresponding edge and/or corner; see Figure~\ref{fig:Refinement} for the corner-edge example.

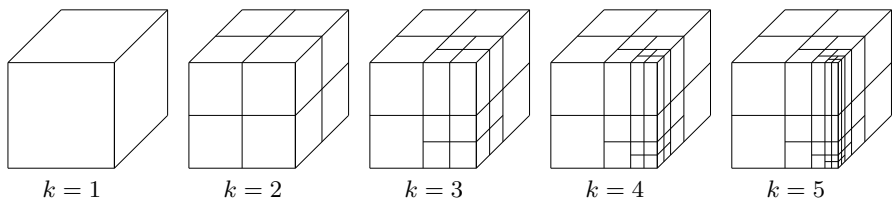
\begin{figure}
\begin{center}
\begin{pspicture}(11.62,2.8)
\psset{linewidth=.01, yunit=1.4, xunit=1.4}
\rput(0,.5){
\psline(0, 0) (1, 0) \psline(0, 0) (0, 1)
\psline(0, 1) (1, 1) \psline(1, 0) (1, 1)
\psline(1, 0) (1.5, .5)
\psline(1, 1) (1.5, 1.5)
\psline(1, 0) (1, 1)
\psline(0, 1) (.5, 1.5)
\psline(.5, 1.5) (1.5, 1.5)
\psline(1.5, .5) (1.5, 1.5)
}
\rput(0.6,.3){\small{$k=1$}}

\rput(1.7,.5){
\psline(0, 0) (1, 0) \psline(0, 0) (0, 1)
\psline(0, 1) (1, 1) \psline(1, 0) (1, 1)
\psline(1, 0) (1.5, .5)
\psline(1, 1) (1.5, 1.5)
\psline(1, 0) (1, 1)
\psline(0, 1) (.5, 1.5)
\psline(.5, 1.5) (1.5, 1.5)
\psline(1.5, .5) (1.5, 1.5)
\psline(.5, 0) (.5, 1)
\psline(.25, 1.25) (1.25, 1.25)
\psline(.5, 1) (1, 1.5)
\psline(1.25, .25) (1.25, 1.25)
\psline(0, .5) (1, .5)
\psline(1, .5) (1.5, 1)
}
\rput(2.3,0.3){\small{$k=2$}}

\rput(3.4,.5){
\psline(0, 0) (1, 0) \psline(0, 0) (0, 1)
\psline(0, 1) (1, 1) \psline(1, 0) (1, 1)
\psline(1, 0) (1.5, .5)
\psline(1, 1) (1.5, 1.5)
\psline(1, 0) (1, 1)
\psline(0, 1) (.5, 1.5)
\psline(.5, 1.5) (1.5, 1.5)
\psline(1.5, .5) (1.5, 1.5)
\psline(.5, 0) (.5, 1)
\psline(.25, 1.25) (1.25, 1.25)
\psline(.5, 1) (1, 1.5)
\psline(1.25, .25) (1.25, 1.25)
\psline(0, .5) (1, .5)
\psline(1, .5) (1.5, 1)
\psline(.75, 0) (.75, 1)
\psline(1.125, 0.125) (1.125, 1.125)
\psline(.625, 1.125) (1.125, 1.125)
\psline(.75, 1) (1, 1.25)
\psline(.5, .25) (1, .25)
\psline(1, .25) (1.25, .5)
}
\rput(4,0.3){\small{$k=3$}}

\rput(5.1,.5){
\psline(0, 0) (1, 0) \psline(0, 0) (0, 1)
\psline(0, 1) (1, 1) \psline(1, 0) (1, 1)
\psline(1, 0) (1.5, .5)
\psline(1, 1) (1.5, 1.5)
\psline(1, 0) (1, 1)
\psline(0, 1) (.5, 1.5)
\psline(.5, 1.5) (1.5, 1.5)
\psline(1.5, .5) (1.5, 1.5)
\psline(.5, 0) (.5, 1)
\psline(.25, 1.25) (1.25, 1.25)
\psline(.5, 1) (1, 1.5)
\psline(1.25, .25) (1.25, 1.25)
\psline(0, .5) (1, .5)
\psline(1, .5) (1.5, 1)
\psline(.75, 0) (.75, 1)
\psline(1.125, 0.125) (1.125, 1.125)
\psline(.625, 1.125) (1.125, 1.125)
\psline(.75, 1) (1, 1.25)
\psline(.5, .25) (1, .25)
\psline(1, .25) (1.25, .5)
\psline(.875, 1) (1, 1.125)
\psline(.8125, 1.0625) (1.0625, 1.0625)
\psline(.875, 0) (.875, 1)
\psline(1.0625, .0625) (1.0625, 1.0625)
\psline(.75, .125) (1, .125)
\psline(1, .125) (1.125, .25)
}
\rput(5.7,0.3){\small{$k=4$}}

\rput(6.8,.5){
\psline(0, 0) (1, 0) \psline(0, 0) (0, 1)
\psline(0, 1) (1, 1) \psline(1, 0) (1, 1)
\psline(1, 0) (1.5, .5)
\psline(1, 1) (1.5, 1.5)
\psline(1, 0) (1, 1)
\psline(0, 1) (.5, 1.5)
\psline(.5, 1.5) (1.5, 1.5)
\psline(1.5, .5) (1.5, 1.5)
\psline(.5, 0) (.5, 1)
\psline(.25, 1.25) (1.25, 1.25)
\psline(.5, 1) (1, 1.5)
\psline(1.25, .25) (1.25, 1.25)
\psline(0, .5) (1, .5)
\psline(1, .5) (1.5, 1)
\psline(.75, 0) (.75, 1)
\psline(1.125, 0.125) (1.125, 1.125)
\psline(.625, 1.125) (1.125, 1.125)
\psline(.75, 1) (1, 1.25)
\psline(.5, .25) (1, .25)
\psline(1, .25) (1.25, .5)
\psline(.875, 1) (1, 1.125)
\psline(.8125, 1.0625) (1.0625, 1.0625)
\psline(.875, 0) (.875, 1)
\psline(1.0625, .0625) (1.0625, 1.0625)
\psline(.75, .125) (1, .125)
\psline(1, .125) (1.125, .25)
\psline(.9375, 0) (.9375, 1)
\psline(1.03125, 0.03125) (1.03125, 1.03125)
\psline(.90625, 1.03125) (1.03125, 1.03125)
\psline(.9375, 1) (1, 1.0625)
\psline(.875, .0625) (1, .0625)
\psline(1, .0625) (1.0625, .125)
}
\rput(7.4,.3){\small{$k=5$}}
\end{pspicture}
\end{center}
\vspace*{-15pt}
\caption{The mesh and the approximation degree during the first 4 refinement
steps for the approximation of the solution with a corner-edge singularity.}
\label{fig:Refinement}
\end{figure}

In Figures~\ref{fig:Edge} and~\ref{fig:CornerEdge} we display the error of the
approximation in the DG-norm~\eqref{eq:DGNorm} in a semi-logarithmic coordinate system with respect to the 5th, respectively the 4th root of the number of degrees of freedom $N$; cf.~\eqref{eq:ExponentialConvergence}. Indeed, on a single element, the number of degrees of freedom grows with~$\mathcal{O}(k^3)$; in addition, when resolving a corner-edge singularity, the number of elements grows with~$\mathcal{O}(\ell^2)$, while, for an edge or a corner singularity, only $\mathcal{O}(\ell)$ many elements are needed. Hence, recalling that~$k\sim\ell$, in the cases of the edge and the corner singularity examples a growth of~$\mathcal{O}(N^4)$ degrees of freedom is obtained, while in the case of the corner-edge example we even have~$\mathcal{O}(N^5)$ (thus the 5th root in~\eqref{eq:ExponentialConvergence}). The graphs show that, after some initial refinement steps, we obtain nearly constant slopes in all three situations. Hence, these experiments confirm that the proposed spectral DGFEM~\eqref{eq:mixedDG} on geometric edge meshes is able to resolve isotropic as well as anisotropic singularities at exponential rates. 

\begin{figure}
\begin{center}
\includegraphics[width=0.49\textwidth]{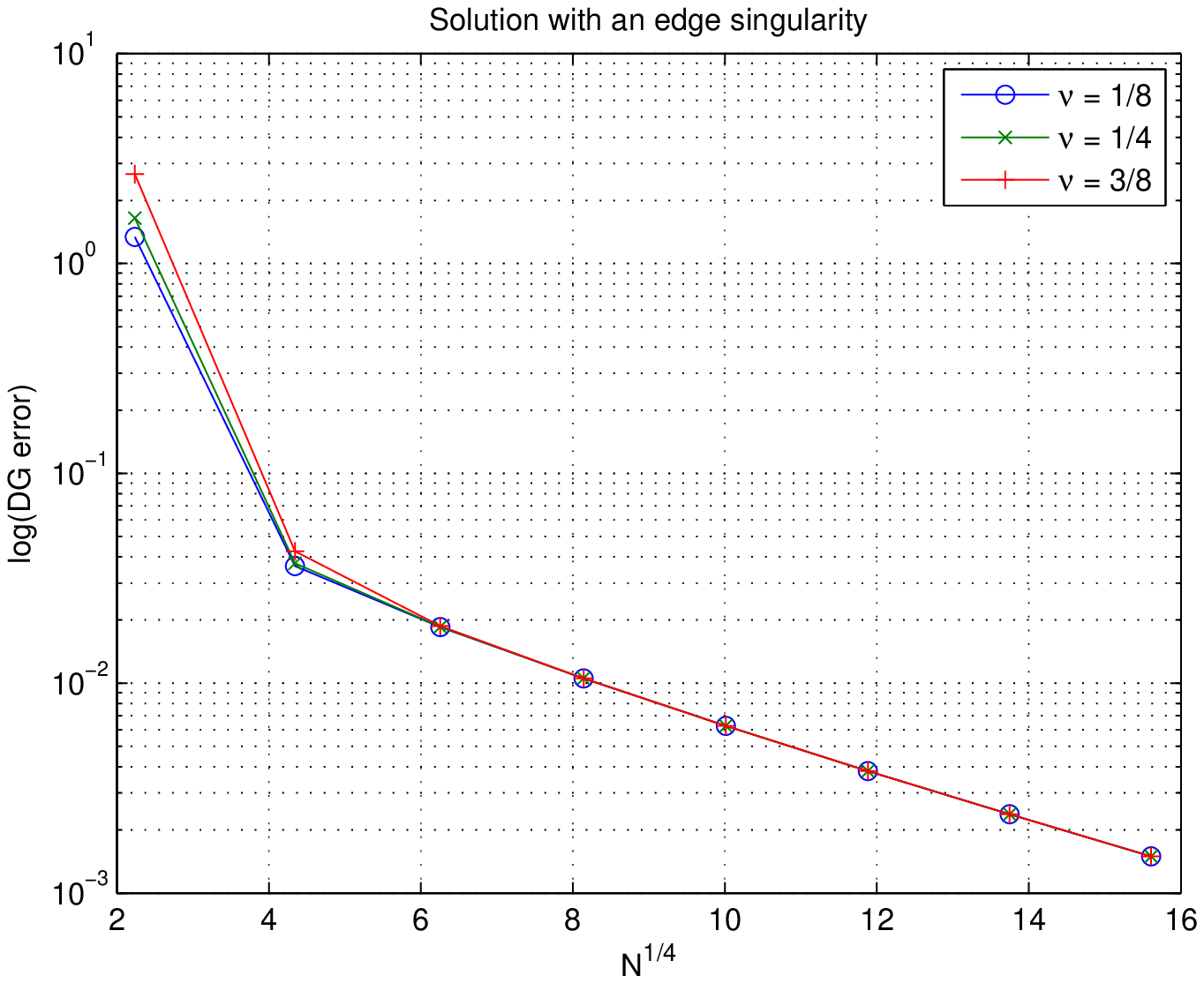}
\includegraphics[width=0.49\textwidth]{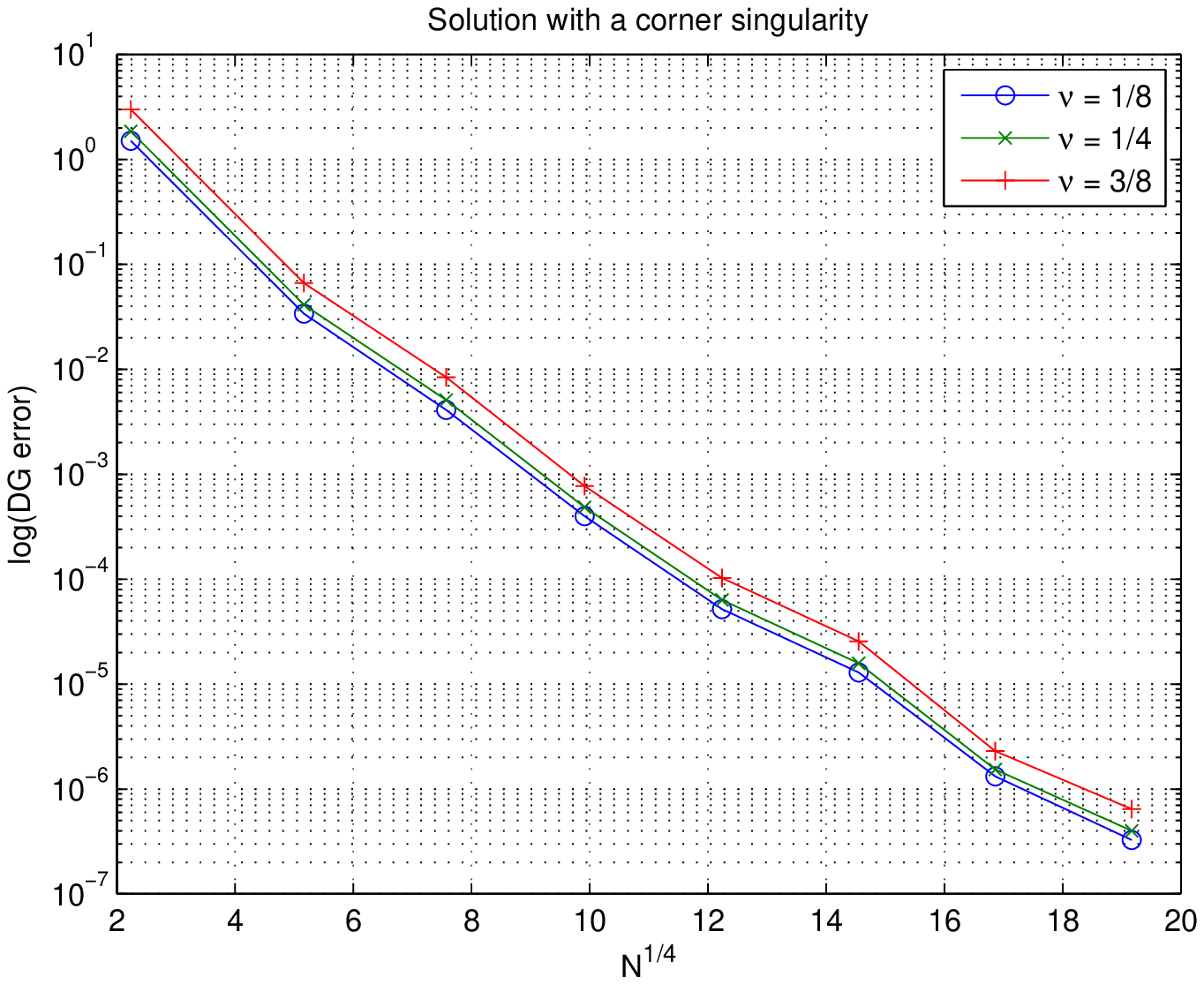}
\end{center}
\caption{Performance of the DGFEM for the solutions with an edge
singularity (left) and a corner singularity (right).}
\label{fig:Edge}
\end{figure}

\begin{figure}
\begin{center}
\includegraphics[width=0.49\textwidth]{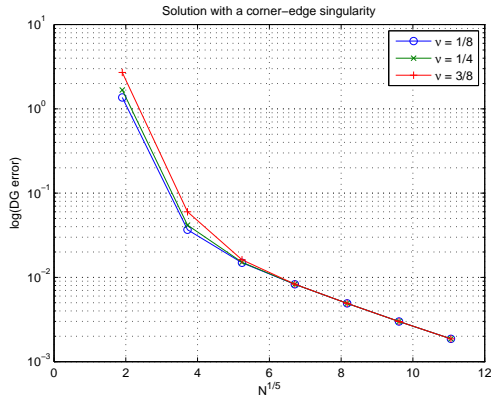}
\end{center}
\vspace*{-15pt}
\caption{Performance of the DGFEM for the solution with a
corner-edge singularity.}
\label{fig:CornerEdge}
\end{figure}

\subsection{Robustness with respect to the Poisson ratio}
The purpose of the second series of experiments is to investigate the robustness of the exponential convergence bound~\eqref{eq:ExponentialConvergence} with respect to~$\nu$ as~$\nu\to\nicefrac12$. The domain $\Omega$ is again chosen to be the unit cube.

\begin{example}\label{ex:1}
We first consider an example where the displacement~$\bm u$ is smooth and divergence-free:
\[
\uu u = 
\sin(\pi x)\sin(\pi y)\sin(\pi z)\cdot
\begin{pmatrix}
\sin(\pi x)\cos(\pi
y)\cos(\pi z) \\
\sin(\pi y)\cos(\pi
x)\cos(\pi z) \\
-2\sin(\pi z)\cos(\pi
x)\cos(\pi y)
\end{pmatrix}.
\]
In this case it immediately follows that~$p=0$, and, hence,  $-\Delta\uu u=\uu f$; in particular, the resulting right-hand side force function~$\bm f$ is independent of~$\nu$. 

For this example, we use a fixed uniform mesh consisting of 64 elements, and simply vary the uniform polynomial degree~$k$. For this setup, since the solution~$(\bm u,p)$ is analytic, we expect exponential convergence with respect to the 3rd root of the number of degrees of freedom. In order to study the convergence with respect~$\nu$ as $\nu\to\nicefrac{1}{2}$, in Figure~\ref{fig:StokesConvergence}, we plot the error of the DG method with respect to different values of the Poisson ratio~$\nu$. Clearly, the deviations between the different exponential convergence curves are almost negligible, and, thereby, underline the robustness of the DGFEM with respect to $\nu$ for this example. 

\begin{figure}
\begin{center}
\includegraphics[width=0.49\textwidth]{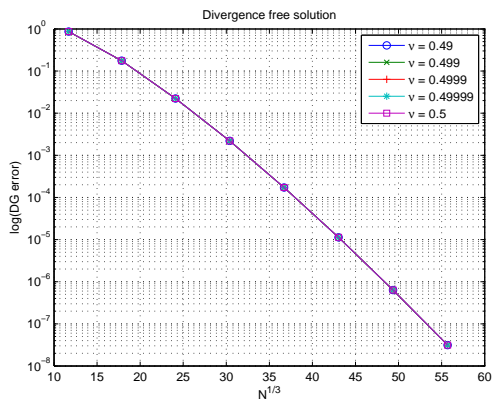}
\end{center}
\caption{Example~\ref{ex:1}: Performance of the spectral DGFEM for different values of~$\nu$.}
\label{fig:StokesConvergence}
\end{figure}
\end{example}

\begin{example}\label{ex:2}
In our last experiment, we choose a circular force in the $x$-$y$-plane and a
linear force in the $z$-direction, i.e.
\[
\uu f = \begin{pmatrix}
-y-\nicefrac12 \\ 
x-\nicefrac12\\
x-\nicefrac12
\end{pmatrix}.
\]
Since the exact solution is not known in this example, we compute a reference
solution based on refining all edges and corners of $\Omega$ with~$k=\ell=5$;
cf. Figure~\ref{fig:RefinementAll}. 
The DG error for different values of~$\nu$ is depicted in Figure~\ref{fig:ReferenceSolution}; as in the previous example, we observe that the DGFEM remains stable when $\nu$ tends to the incompressible limit~$\nicefrac12$. Furthermore, the nearly straight graphs indicate that exponential convergence is also achieved in these computations.
\end{example}

\begin{figure}
\begin{center}
\begin{pspicture}(0,0.3)(11.62,2.9)
\psset{linewidth=.01, yunit=1.4, xunit=1.4}
\rput(0,.5){
\psline(0, 0) (1, 0) \psline(0, 0) (0, 1)
\psline(0, 1) (1, 1) \psline(1, 0) (1, 1)
\psline(1, 0) (1.5, .5)
\psline(1, 1) (1.5, 1.5)
\psline(1, 0) (1, 1)
\psline(0, 1) (.5, 1.5)
\psline(.5, 1.5) (1.5, 1.5)
\psline(1.5, .5) (1.5, 1.5)
}
\rput(0.6,.3){\small{$k=1$}}

\rput(1.7,.5){
\psline(0, 0) (1, 0) \psline(0, 0) (0, 1)
\psline(0, 1) (1, 1) \psline(1, 0) (1, 1)
\psline(1, 0) (1.5, .5)
\psline(1, 1) (1.5, 1.5)
\psline(1, 0) (1, 1)
\psline(0, 1) (.5, 1.5)
\psline(.5, 1.5) (1.5, 1.5)
\psline(1.5, .5) (1.5, 1.5)
\psline(.5, 0) (.5, 1)
\psline(.25, 1.25) (1.25, 1.25)
\psline(.5, 1) (1, 1.5)
\psline(1.25, .25) (1.25, 1.25)
\psline(0, .5) (1, .5)
\psline(1, .5) (1.5, 1)

}
\rput(2.3,0.3){\small{$k=2$}}

\rput(3.4,.5){
\psline(0, 0) (1, 
   0) \psline(0, .25) (1, .25) \psline(0, .5) (1, .5)
\psline(0, .75) (1, .75) \psline(0, 1) (1, 1)
\psline(0, 0) (0, 1) \psline(.25, 0) (.25, 1) \psline(.5, 
   0) (.5, 1)
\psline(.75, 0) (.75, 1) \psline(1, 0) (1, 1)
\psline(1, 
   0) (1.5, .5) \psline(1, .25) (1.5, .75) \psline(1, .5) (1.5,
    1)
\psline(1, .75) (1.5, 1.25) \psline(1, 1) (1.5, 1.5)
\psline(1.125, .125) (1.125, 1.125) \psline(1.25, .25) (1.25, 
   1.25)
\psline(1.375, .375) (1.375, 1.375) \psline(1.5, .5) (1.5, 1.5)
\psline(.125, 1.125) (1.125, 1.125) \psline(.25, 1.25) (1.25, 
   1.25)
\psline(.375, 1.375) (1.375, 1.375) \psline(.5, 1.5) (1.5, 1.5)
\psline(0, 1) (.5, 1.5) \psline(.25, 1) (.75, 1.5) \psline(.5, 
   1) (1, 1.5)
\psline(.75, 1) (1.25, 1.5)

}
\rput(4,0.3){\small{$k=3$}}

\rput(5.1,.5){
\psline(0, 0) (1, 
   0) \psline(0, .25) (1, .25) \psline(0, .5) (1, .5)
\psline(0, .75) (1, .75) \psline(0, 1) (1, 1)
\psline(0, 0) (0, 1) \psline(.25, 0) (.25, 1) \psline(.5, 
   0) (.5, 1)
\psline(.75, 0) (.75, 1) \psline(1, 0) (1, 1)
\psline(1, 
   0) (1.5, .5) \psline(1, .25) (1.5, .75) \psline(1, .5) (1.5,
    1)
\psline(1, .75) (1.5, 1.25) \psline(1, 1) (1.5, 1.5)
\psline(1.125, .125) (1.125, 1.125) \psline(1.25, .25) (1.25, 
   1.25)
\psline(1.375, .375) (1.375, 1.375) \psline(1.5, .5) (1.5, 1.5)
\psline(.125, 1.125) (1.125, 1.125) \psline(.25, 1.25) (1.25, 
   1.25)
\psline(.375, 1.375) (1.375, 1.375) \psline(.5, 1.5) (1.5, 1.5)
\psline(0, 1) (.5, 1.5) \psline(.25, 1) (.75, 1.5) \psline(.5, 
   1) (1, 1.5)
\psline(.75, 1) (1.25, 1.5)

\psline(0, .125) (1, .125)
\psline(0, .875) (1, .875)
\psline(.125, 0) (.125, 1)
\psline(.875, 0) (.875, 1)

\psline(1, .125) (1.5, .625)
\psline(1, .875) (1.5, 1.375)
\psline(1.0625, .0625) (1.0625, 1.0625)
\psline(1.4375, .4375) (1.4375, 1.4375)

\psline(.0625, 1.0625) (1.0625, 1.0625)
\psline(.4375, 1.4375) (1.4375, 1.4375)
\psline(.125, 1) (.625, 1.5)
\psline(.875, 1) (1.375, 1.5)

}
\rput(5.7,0.3){\small{$k=4$}}

\rput(6.8,.5){
\psline(0, 0) (1, 
   0) \psline(0, .25) (1, .25) \psline(0, .5) (1, .5)
\psline(0, .75) (1, .75) \psline(0, 1) (1, 1)
\psline(0, 0) (0, 1) \psline(.25, 0) (.25, 1) \psline(.5, 
   0) (.5, 1)
\psline(.75, 0) (.75, 1) \psline(1, 0) (1, 1)
\psline(1, 
   0) (1.5, .5) \psline(1, .25) (1.5, .75) \psline(1, .5) (1.5,
    1)
\psline(1, .75) (1.5, 1.25) \psline(1, 1) (1.5, 1.5)
\psline(1.125, .125) (1.125, 1.125) \psline(1.25, .25) (1.25, 
   1.25)
\psline(1.375, .375) (1.375, 1.375) \psline(1.5, .5) (1.5, 1.5)
\psline(.125, 1.125) (1.125, 1.125) \psline(.25, 1.25) (1.25, 
   1.25)
\psline(.375, 1.375) (1.375, 1.375) \psline(.5, 1.5) (1.5, 1.5)
\psline(0, 1) (.5, 1.5) \psline(.25, 1) (.75, 1.5) \psline(.5, 
   1) (1, 1.5)
\psline(.75, 1) (1.25, 1.5)

\psline(0, .125) (1, .125)
\psline(0, .875) (1, .875)
\psline(.125, 0) (.125, 1)
\psline(.875, 0) (.875, 1)

\psline(0, .0625) (1, .0625)
\psline(0, .9375) (1, .9375)
\psline(.0625, 0) (.0625, 1)
\psline(.9375, 0) (.9375, 1)

\psline(1, .125) (1.5, .625)
\psline(1, .875) (1.5, 1.375)
\psline(1.0625, .0625) (1.0625, 1.0625)
\psline(1.4375, .4375) (1.4375, 1.4375)

\psline(1, .0625) (1.5, .5625)
\psline(1, .9375) (1.5, 1.4375)
\psline(1.03125, .03125) (1.03125, 1.03125)
\psline(1.46875, .46875) (1.46875, 1.46875)

\psline(.0625, 1.0625) (1.0625, 1.0625)
\psline(.4375, 1.4375) (1.4375, 1.4375)
\psline(.125, 1) (.625, 1.5)
\psline(.875, 1) (1.375, 1.5)

\psline(.03125, 1.03125) (1.03125, 1.03125)
\psline(.46875, 1.46875) (1.46875, 1.46875)
\psline(.0625, 1) (.5625, 1.5)
\psline(.9375, 1) (1.4375, 1.5)

}
\rput(7.4,.3){\small{$k=5$}}
\end{pspicture}
\end{center}
\caption{The mesh and the approximation degree $k$ for the first five
refinement steps for the approximation of the reference solution.}
\label{fig:RefinementAll}
\end{figure}
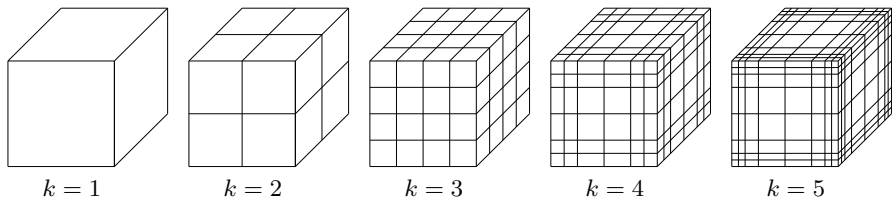

\begin{figure}
\begin{center}
\includegraphics[width=0.49\textwidth]{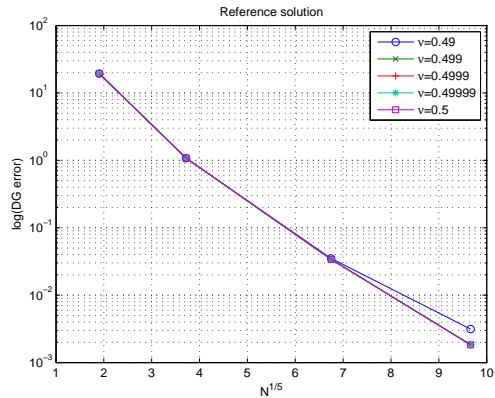}
\end{center}
\caption{Example~\ref{ex:2}: Performance of the spectral DGFEM for different values of~$\nu$.}
\label{fig:ReferenceSolution}
\end{figure}


\section{Conclusions}\label{sec:conclusions}
This paper centres on spectral mixed discontinuous
Galerkin discretizations for linear elasticity and Stokes flow in
three dimensional polyhedral domains. In a series of numerical experiments we have validated our theoretical results from~\cite{wihlerWirz} for various canonical reference situations. In particular, we have performed a computational study on the inf-sup stability and the exponential convergence of this class of methods on anisotropic geometric edge meshes. For the former purpose we have derived a simple procedure to determine the discrete inf-sup constants based on a singular value decomposition approach along the lines of~\cite{BrezziFortin91}.

Following the approach~\cite{SchoetzauSchwabWihler2009:1,SchoetzauSchwabWihler2009:2,SchoetzauSchwabWihlerNeumann}, our work may be extended to variable (and possibly anisotropic) polynomial degree distributions, thereby leading to $hp$-version DGFEM. To prove stability results in that context, however, the discrete inf-sup condition~\eqref{eq:disc_infsup} would need to be generalized to the corresponding $hp$-meshes. Finally, let us mention that the linear mixed DG discretizations addressed in the present work could be combined with the iterative Newton DG (NDG) approach~\cite{HoustonWihler:2017} in order to approximate problems in nonlinear elasticity in three dimensions.


\bibliographystyle{amsplain}
\bibliography{literature}

\providecommand{\bysame}{\leavevmode\hbox to3em{\hrulefill}\thinspace}
\providecommand{\MR}{\relax\ifhmode\unskip\space\fi MR }
\providecommand{\MRhref}[2]{%
  \href{http://www.ams.org/mathscinet-getitem?mr=#1}{#2}
}
\providecommand{\href}[2]{#2}
\begin{thebibliography}{10}

\bibitem{ArnoldBrezziCockburnMarini2001}
D.~N. Arnold, F.~Brezzi, B.~Cockburn, and L.~D. Marini, \emph{Unified analysis
  of discontinuous {G}alerkin methods for elliptic problems}, SIAM J. Numer.
  Anal. \textbf{39} (2001), 1749--1779.

\bibitem{BG96}
I.~Babu\v{s}ka and B.~Q. Guo, \emph{Approximation properties of the {$h$}-{$p$}
  version of the finite element method}, Comput. Methods Appl. Mech. Engrg.
  \textbf{133} (1996), no.~3-4, 319--346.

\bibitem{BangerthHartmannKanschat2007}
W.~Bangerth, R.~Hartmann, and G.~Kanschat, \emph{{deal.II} -- a general purpose
  object oriented finite element library}, ACM Trans. Math. Softw. \textbf{33}
  (2007), no.~4, 24/1--24/27.

\bibitem{DealIIReference}
W.~Bangerth, T.~Heister, G.~Kanschat, et~al., \emph{{\tt deal.{I}{I}}
  differential equations analysis library, technical reference},
  \texttt{http://www.dealii.org}.

\bibitem{BrezziFortin91}
F.~Brezzi and M.~Fortin, \emph{Mixed and hybrid finite element methods},
  {S}pringer {S}eries in {C}omputational {M}athematics, vol.~15, Springer, New
  York, 1991.

\bibitem{DaugeCostabelNicaise}
M.~Dauge, M.~Costabel, and S.~Nicaise, \emph{Analytic regularity for linear
  elliptic systems in polygons and polyhedra}, Mathematical Models and Methods
  in Applied Sciences \textbf{22} (2012), no.~8, 1250015.

\bibitem{GeHaHo07}
E.~H. Georgoulis, E.~Hall, and P.~Houston, \emph{Discontinuous {G}alerkin
  methods on {$hp$}-anisotropic meshes. {I}. {A} priori error analysis}, Int.
  J. Comput. Sci. Math. \textbf{1} (2007), no.~2-4, 221--244.

\bibitem{GiraultRaviart86}
V.~Girault and P.~A. Raviart, \emph{Finite element methods for
  {N}avier--{S}tokes equations}, Springer, New York, 1986.

\bibitem{Guo95}
B.~Q. Guo, \emph{The $h$-$p$ version of the finite element method for solving
  boundary value problems in polyhedral domains}, Boundary Value Problems and
  Integral Equations in Nonsmooth Domains, Lecture Notes in Pure and Applied
  Mathematics, vol. 167, Dekker, New York, 1995, pp.~101--120.

\bibitem{BabGuo_edge1}
B.~Q. Guo and I.~Babu\v{s}ka, \emph{Regularity of the solutions for elliptic
  problems on nonsmooth domains in $\mathbb{R}\sp 3$. {I}. {C}ountably normed
  spaces on polyhedral domains}, Proc. Roy. Soc. Edinburgh Sect. A \textbf{127}
  (1997), no.~1, 77--126.

\bibitem{BabGuo_edge2}
\bysame, \emph{Regularity of the solutions for elliptic problems on nonsmooth
  domains in $\mathbb{R}\sp 3$. {II}. {R}egularity in neighbourhoods of edges},
  Proc.~Roy.~Soc.~Edinburgh Sect.~A \textbf{127} (1997), no.~3, 517--545.

\bibitem{HoustonSchoetzauWihler2006}
P.~Houston, D.~Sch{\"o}tzau, and T.~P. Wihler, \emph{An {$hp$}-adaptive mixed
  discontinuous {G}alerkin {FEM} for nearly incompressible linear elasticity},
  Comput. Methods Appl. Mech. Engrg. \textbf{195} (2006), no.~25-28,
  3224--3246.

\bibitem{HoustonWihler:2017}
P.~Houston and T.~P. Wihler, \emph{An $hp$-adaptive
  {N}ewton-discontinuous-{G}alerkin finite element approach for semilinear
  elliptic boundary value problems}, Math. Comp. \textbf{87} (2018), no.~314,
  2641--2674.

\bibitem{MadayBernardi}
Y.~Maday and C.~Bernardi, \emph{Uniform inf--sup conditions for the spectral
  discretization of the {S}tokes problem}, Mathematical Models and Methods in
  Applied Sciences \textbf{09} (1999), no.~03, 395--414.

\bibitem{MadayMarcati:2019}
Y.~Maday and C.~Marcati, \emph{Regularity and {$hp$} discontinuous {G}alerkin
  finite element approximation of linear elliptic eigenvalue problems with
  singular potentials}, Math. Models Methods Appl. Sci. \textbf{29} (2019),
  no.~8, 1585--1617.

\bibitem{MaRo10}
V.~Maz'ya and J.~Rossmann, \emph{Elliptic equations in polyhedral domains},
  Mathematical Surveys and Monographs, vol. 162, American Mathematical Society,
  Providence, RI, 2010.

\bibitem{MazzucatoNistor}
A.~L. Mazzucato and V.~Nistor, \emph{Well-posedness and regularity for the
  elasticity equation with mixed boundary conditions on polyhedral domains and
  domains with cracks}, Arch. Ration. Mech. Anal. \textbf{195} (2010), no.~1,
  25--73.

\bibitem{Schoetzau:2001:MHP}
D.~Sch{\"o}tzau, C.~Schwab, and A.~Toselli, \emph{Mixed {$hp$}-{DGFEM} for
  incompressible flows}, SIAM J. Numer. Anal. \textbf{40} (2003), 2171--2194.

\bibitem{Schoetzau:2004:MHP}
\bysame, \emph{Mixed {$hp$}-{DGFEM} for incompressible flows. {II}. {G}eometric
  edge meshes}, IMA J. Numer. Anal. \textbf{24} (2004), no.~2, 273--308.

\bibitem{SchoetzauSchwabWihler2009:1}
D.~Sch{\"o}tzau, C.~Schwab, and T.~P. Wihler, \emph{{$hp$}-d{GFEM} for
  second-order elliptic problems in polyhedra {I}: {S}tability on geometric
  meshes}, SIAM J. Numer. Anal. \textbf{51} (2013), no.~3, 1610--1633.
  \MR{3061472}

\bibitem{SchoetzauSchwabWihler2009:2}
\bysame, \emph{{$hp$}-d{GFEM} for second order elliptic problems in polyhedra
  {II}: {E}xponential convergence}, SIAM J. Numer. Anal. \textbf{51} (2013),
  no.~4, 2005--2035. \MR{3073653}

\bibitem{SchoetzauSchwabWihlerNeumann}
\bysame, \emph{{$hp$-DGFEM} for second-order mixed elliptic problems in
  polyhedra}, Math. Comp. \textbf{85} (2016), no.~299, 1051--1083.

\bibitem{SW03}
D.~Sch{\"o}tzau and T.~P. Wihler, \emph{Exponential convergence of mixed
  {$hp$}-{DGFEM} for {S}tokes flow in polygons}, Numer. Math. \textbf{96}
  (2003), 339--361.

\bibitem{wihlerWirz}
T.~P. Wihler and M.~Wirz, \emph{Mixed hp-discontinuous {G}alerkin {FEM} for
  linear elasticity and {S}tokes flow in three dimensions}, Mathematical Models
  and Methods in Applied Sciences \textbf{22} (2012), no.~8, 1250016.

\end{thebibliography}

\end{document}